%% file: main.tex
\documentclass[10pt,twocolumn,twoside]{IEEEtran}
\usepackage{generic}
\usepackage{cite}
\usepackage{amsmath,amssymb,amsfonts}
\usepackage{algorithmic}
\usepackage{graphicx}
\usepackage{textcomp}

\usepackage{lmodern,microtype,bm}     
\usepackage{dsfont, graphicx, color,scalerel}
\usepackage{mathrsfs}
\usepackage[utf8]{inputenc}

\usepackage{amsthm, thmtools,thm-restate}
\usepackage{cuted}
\usepackage{csquotes}

\declaretheorem[name=Theorem]{mytheo}
\declaretheorem[name=Lemma]{mylem}
\declaretheorem[name=Proposition]{myprop}
\declaretheorem[name=Definition]{mydef}

\declaretheorem[name=Assumption]{myas}

\newcommand{\abs}[1]{\left\vert{#1}\right\vert}

\newcommand{\vo}[1]{\boldsymbol{#1}}
\newcommand{\cj}[1]{\overline{{#1}}}

\newcommand{\ip}[2]{\left\langle{#1},{#2}\right\rangle}
\newcommand{\lMAX}{\lambda}
\newcommand{\lMIN}{\varsigma}
\newcommand{\lmax}[1]{\lambda_{{#1}}}
\newcommand{\lmin}[1]{\varsigma_{{#1}}}

\newcommand{\fmax}[1]{\Phi^+_{{#1}}}
\newcommand{\fmin}[1]{\Phi^-_{{#1}}}

\newcommand{\pmax}[1]{\Psi^+_{{#1}}}
\newcommand{\pmin}[1]{\Psi^-_{{#1}}}
\newcommand{\nMAX}{\nu^+}
\newcommand{\nMIN}{\nu^-}
\newcommand{\nmax}[1]{\nu^+_{{#1}}}
\newcommand{\nmin}[1]{\nu^-_{{#1}}}
\newcommand{\cjnmax}[1]{\cj{\nu}^+_{{#1}}}
\newcommand{\cjnmin}[1]{\cj{\nu}^-_{{#1}}}
\newcommand{\acrit}[1]{\alpha_{{#1}}}
\newcommand{\bcrit}[1]{\beta_{{#1}}}

\def\BibTeX{{\rm B\kern-.05em{\sc i\kern-.025em b}\kern-.08em
    T\kern-.1667em\lower.7ex\hbox{E}\kern-.125emX}}
\begin{document}
\title{Collective rhythm design in coupled mixed-feedback systems through dominance and bifurcations}
\author{Omar Juarez-Alvarez, Alessio Franci
\thanks{This work was supported by the Belgian Government through the FPS Policy and Support (BOSA) grant NEMODEI.}
\thanks{Omar Juarez-Alvarez is with the Department of Mathematics, Universidad Nacional Autónoma de México, Ciudad de México, México (e-mail: pat\_jualv@ciencias.unam.mx).}
\thanks{Alessio Franci is with 
the Montefiore Institute, University of Liege, Liege, Belgium, with the Department of Mathematics, Universidad Nacional Autónoma de México, Ciudad de México, México, and with the WEL Research Institute, Wavre, Belgium (e-mail: alessio.franci83@gmail.com).}}

\maketitle

\begin{abstract}
The theory of mixed-feedback systems provides an effective framework for the design of robust and tunable oscillations in nonlinear systems characterized by interleaved fast positive and slow negative feedback loops. The goal of this paper is to extend the mixed-feedback oscillation design framework to networks. To this aim, we introduce a network model of coupled mixed-feedback systems, ask under which conditions it exhibits a collective oscillatory rhythm, and if, and how, this rhythm can be shaped by network design. In the proposed network model, node dynamics are nonlinear and defined by a tractable realization of the mixed-feedback structure. Coupling between nodes is also nonlinear and defined by a tractable abstraction of synaptic coupling between neurons. We derive constructive conditions under which the spectral properties of the network adjacency matrix fully and explicitly determine both the emergence of a stable network rhythm and its detailed rhythmic profile, i.e., the pattern of relative oscillation amplitudes and phase differences. Our theoretical developments are grounded on ideas from dominant systems and bifurcation theory. They provide a new framework for the analysis and design of nonlinear network rhythms.
\end{abstract}


\section{Introduction}
\label{sec:introduction}

The theory of {\it mixed-feedback systems} has been recently developed to describe and understand excitable, spiking, and rhythmic biological behaviors, as well as a tool to design such behaviors in artificial systems~\cite{sepulchre2019control,sepulchre2018excitable,Franci2018,ribar2019neuromodulation}. Mixed-feedback systems are identified by the co-existence of negative and positive feedback loops at different timescales. As such, they inherit and merge the best of both positive and negative feedback control: they combine the robust reliability of digital automata, rooted in positive feedback control, and the adaptive flexibility of analog systems, rooted in negative feedback control. From cell cycles~\cite{tsai2008robust} and circadian rhythms~\cite{smolen2001modeling}, to neuron electrical activity~\cite{drion2015neuronal} and the dynamics of cortical neural circuits~\cite{douglas1991functional}, the mixed-feedback structure is a hallmark of adaptive biological behaviors across scale. The mixed-feedback theory developed in~\cite{sepulchre2019control,sepulchre2018excitable,franci2019sensitivity,franci2014modeling} has paved the way to design similar adaptive, bio-inspired behaviors in engineered systems, including neuromorphic or ``spiking'' control systems~\cite{sepulchre2022spiking} and their embodiment in simple robots~\cite{bartolozzi2022embodied,cathcart2024excitable}.

In the recent work~\cite{che2023dominant}, the authors derived both graphical and algebraic criteria to ensure the existence of robust and tunable oscillations in a {\it single} mixed-feedback system. The goal of this paper is to undertake the same program for {\it networks of coupled mixed-feedback systems}. In biology, mixed-feedback systems such as neurons or molecular clocks are ubiquitously interconnected in networks for coordinated activity. Examples are neural central pattern generators~\cite{marder2001} and distributed circadian clocks~\cite{noguchi17,yamaguchi03,juarez23}.
Yet, a network theory of coupled mixed-feedback system is still lacking.

The first contribution of the paper is to introduce a mathematically tractable network model of coupled mixed-feedback systems. Node dynamics are defined by a simple two-dimensional, fast-slow, mixed-feedback loop. These dynamics are akin to other simple dynamics used to model different collective behaviors, e.g., the first-order integrator of consensus dynamics~\cite{moreau2005stability,olfati2004consensus} or the phase model of weakly coupled oscillator networks~\cite{kuramoto1984chemical,izhikevich2008phase}. Furthermore, they can readily be translated to analog electronics for embodied intelligence applications~\cite{ribar2019neuromodulation}. Coupling between mixed-feedback systems happens through saturated additive interactions mediated by the fast variables of each node. This form of coupling is highly reminiscent of synaptic coupling between biological neurons~\cite{abbott1998modeling}.

Our second contribution is to formulate a collective rhythm control problem. Namely, we ask if it is possible to design the adjacency matrix of the mixed-feedback network to {\it i)} ensure a stable network rhythm, that is, the convergence of almost all solutions to a limit cycle along which all the nodes oscillate with the same period but possibly different amplitudes and phases, and {\it ii)} design and control the collective rhythmic profile, that is, the specific pattern of relative oscillation amplitudes and phases among the nodes.

The third and main contribution of the paper, which generalizes the preliminary results presented in~\cite{juarez21}, is to show that the collective rhythm control problem can be solved constructively, and to derive such a constructive solution. The derivation of the proposed solution relies fundamentally on the fast-slow nature of mixed-feedback dynamics and on the key assumption that the fast dynamics of the mixed-feedback network satisfy a dominance condition. This condition can easily be checked by inspection of the leading eigenstructure of the network adjacency matrix. Given these two ingredients, we use linear algebraic and bifurcation theory methods to show that, if the network adjacency matrix has a simple real leading eigenvalue or a simple pair of complex conjugate leading eigenvalues, then a network rhythm can be ignited by sufficiently large self-positive feedback at the node level or by sufficiently large coupling strength at the network level. We characterize the rhythm stability and we show that the leading eigenvector of the network adjacency matrix precisely determines the rhythmic profile.

The fourth contribution of the paper is to present a series of algebraic results, instrumental to deriving our solution to the collective rhythm control problem, that reveal a non-trivial but tractable mapping between the eigenstructure of the network adjacency matrix and the Jacobian matrix of the network nonlinear dynamics. This mapping allows us to prove the fundamental result that certain dominance properties of the network fast subsystem are inherited by the full fast-slow network dynamics. The timescale separation between the mixed-feedback fast positive and slow negative loops is again key in proving these results.

Finally, the fifth contribution is to describe one out of many simple methodologies stemming from our results to design arbitrary collective rhythms in networks of coupled mixed-feedback systems. Jointly, these contributions define a novel framework for the design of rich rhythmic behaviors in networks of biological or bio-inspired mixed-feedback systems.

The paper is structured as follows. In Section~\ref{sec:math}, some mathematical preliminaries, definitions and notation are presented. The coupled mixed-feedback systems model is introduced and interpreted in Section~\ref{sec:model}. The collective rhythm control problem is rigorously formulated in Section~\ref{sec:relevance}, together with the graphical notation used to represent its solution throughout the paper, some preliminary remarks on the constructive nature of the proposed solutions, the importance of the mixed-feedback structure, and a non-technical summary to guide the reader through the more technical part of the paper in the subsequent sections. The main (dominance) assumption used to solve the collective rhythm control problem is presented and discussed in Section~\ref{sec: fast dom}. Section~\ref{sec:struc} presents results revealing how the mixed-feedback structure ensures a tractable characterization of the mapping between the spectra of the network adjacency matrix and of the model Jacobian. Section~\ref{sec:dom} uses these results to prove the existence of parameter combinations such that the model Jacobian is singular and how, close to singularity, the Jacobian inherits the leading eigenstructure of the network adjacency matrix. Section~\ref{sec:synchro} finally uses these algebraic results to show that the Jacobian singularity of Section~\ref{sec:dom} corresponds to a Hopf bifurcation at which a collective rhythm emerges; it then derives expressions for the stability of this rhythm, and characterizes its profile in terms of the Jacobian leading eigenstructure. Section~\ref{sec:build} presents a couple of simple, rigorous algorithms to design arbitrary collective rhythms in networks of coupled mixed-feedback systems. A discussion and future research directions are provided in Section~\ref{sec:disc}.
Due to space limitations, all technical proofs have been moved to an Appendix that can be found in the paper extended preprint~\cite{juarezalvarez2024analysis}.

\section{Mathematical preliminaries}\label{sec:math}

Real $N$-dimensional vectors are denoted in bold $\vo{x},\vo{v},\vo{\zeta},\ldots,$ and are defined entry-wise as $\vo{x}=(x_1,x_2,\ldots,x_N)\in\mathds{R}^N$. $\vo{0}_N=(0,\ldots,0)\in\mathds{R}^N$ denotes the zero vector, $\vo{1}_N=(1,\ldots,1)\in\mathds{R}^N$ the all-ones vector, and $\vo{e}_j=(\delta_{jk})_{k=1}^N\in\mathds{R}^N$ canonical vectors, where $\delta_{jk}$ is the Kronecker delta. Complex numbers are either expressed in Cartesian form, $z=a+ib$, with $a\in\mathds{R}$, $b\in\mathds{R}$, or in polar form, $z=\rho e^{i\theta}$, for $\rho\geqslant0$ and $\theta\in\mathds{S}^1$, where $\mathds{S}^1:=\mathds{R}\, \mathrm{mod}\,2\pi$. The conjugate of a complex number $z=a+ib$ is $\cj{z}=a-ib$ and its modulus is $\abs{z}=\sqrt{z\cj{z}}$. Complex vectors $\vo{z}\in\mathds{C}^N$ are represented as $\vo{z}=\vo{a}+i\vo{b}$, where real tuples $\vo{a}$, $\vo{b}$ are the real and imaginary parts, respectively, of complex vector $\vo{z}$. The conjugate of a complex vector, $\bar{\vo{z}}=\vo{a}-i\vo{b}$, is computed entry-wise. A complex vector $\vo{z}\in\mathds{C}^N$ is said to be \emph{modulus-homogeneous} if there exists $\kappa\geqslant0$ such that $\abs{\vo{z}_j}=\kappa$, for all $j\in\{1,\ldots,N\}$. The entry-wise Hadamard product of two complex vectors $\vo{z}$ and $\vo{y}$ is denoted by $\vo{z}\odot\vo{y}\in\mathds{C}^N$ and defined entry-wise by $(\vo{z}\odot\vo{y})_i=z_iy_i$. We define two inner products: the matricial inner product $\vo{v}^t\vo{w}$, for real vectors, and the complex inner product $\ip{\vo{v}}{\vo{w}}=\cj{\vo{v}}^t\vo{w}$, for complex vectors. Two indexed sets $U=\{\vo{u}_j\in\mathds{C}^N:\,j\in\{1,\ldots,k\}\}$ and $V=\{\vo{v}_j\in\mathds{C}^N:\,j\in\{1,\ldots,k\}\}$ form a \emph{biorthogonal system} if for every $n\in\{1,\ldots,k\}$ and $m\in\{1,\ldots,k\}$ it holds that $\ip{\vo{u}_n}{\vo{v}_m}=\delta_{nm}.$

Given a parameterized vector field $\vo{f}(\vo{x};p)$ in $\mathds{R}^n$ which is $k$ times differentiable, and an ordered set $\gamma=\{\vo{v}_1,\ldots,\vo{v}_k\}\subseteq\mathds{R}^n$, the $k$th order directional derivative of $\vo{f}$ along $\gamma$ computed at $(\vo{x},p)$ is defined as
\begin{equation*}
    \begin{split}
        (\mathrm{d}^k\vo{f})&_{\vo{x},p}(\vo{v}_1,\ldots,\vo{v}_k):=\dfrac{\partial\,\,\,\,}{\partial t_1}\ldots\dfrac{\partial\,\,\,\,}{\partial t_k}\vo{f}\left(\vo{x}+\sum_{i=1}^kt_i\vo{v}_i;p\right)\\
        =&\sum\dfrac{\partial^k\vo{f}}{\partial x_{i_1}\ldots\partial x_{i_k}}(\vo{x};p)(\vo{v}_1)_{i_1}\ldots(\vo{v}_k)_{i_k},
    \end{split}
\end{equation*}
where the last sum is computed over all the $k$th order partial derivatives of $f$.

We denote the set of $N\times N$ real and complex square matrices as $\mathds{R}^{N\times N}$ and $\mathds{C}^{N\times N}$, respectively. We denote the transpose of matrix $A$ as $A^t$. Any $N$-tuple $\vo{x}$ is considered as a $N\times1$ column matrix, and its transpose $\vo{x}^t$ as a $1\times N$ row matrix. The zero matrix is denoted by $O_N=(0)_{ij}\in\mathds{R}^{N\times N}$, and the identity matrix by $I_N=(\delta_{ij})\in\mathds{R}^{N\times N}$. A matrix $A=(a_{ij})$ is said to be positive (non-negative) if all of its entries $a_{ij}$ are positive (non-negative). Positive and non-negative vectors are similarly defined. Given a set of complex numbers $\{z_1,\ldots,z_N\}$ we denote $D=\mathrm{diag}(z_1,\ldots,z_N)\in\mathds{C}^{N\times N}$ as the diagonal matrix whose entries are given by $D_{ij}=z_i\delta_{ij}$. A {\it switching matrix} $M$ is a diagonal matrix whose diagonal entries are all either 1 or -1~\cite{bizyaeva22}; observe that $M=M^{-1}$ for any switching matrix. Two matrices $A$, $B$ are said to be \emph{switching equivalent} if there exists a switching matrix $M$ such that $B=M^{-1}AM$; thus, they are similar and cospectral. A matrix is said to be irreducible if it is not similar to an upper-triangular matrix. The\emph {spectrum} of matrix $A$, denoted by $\sigma(A)$, is the collection of all of its eigenvalues (also called $A$-eigenvalues), $\sigma(A)=\{\mu_1,\ldots,\mu_N\}$ (repeated eigenvalues appear with their algebraic multiplicity). We use $\vo{v}\in\mathds{C}^N$ and $\vo{w}\in\mathds{C}^N$ to represent left and right eigenvectors satisfying $\vo{v}^tA=\mu\vo{v}^t$ and $A\vo{w}=\mu\vo{w}$, respectively, for some $\mu\in\sigma(A)$. The spectral radius of matrix $A$ is defined as $\rho(A)=\max\{\abs{\mu}:\,\mu\in\sigma(A)\}$.

An eigenvalue $\mu\in\sigma(A)$ is said to be simple if its algebraic multiplicity is equal to one. An element $\mu^*\in\sigma(A)$ is said to be a \emph{leading eigenvalue} if it is simple and satisfies $\mathrm{Re}(\mu^*)\geqslant\mathrm{Re}(\mu)$ for all $\mu\in\sigma(A)$; an element $\mu^*\in\sigma(A)$ is said to be a \emph{strictly leading eigenvalue} if it is simple and satisfies $\mathrm{Re}(\mu^*)>\mathrm{Re}(\mu)$ for all $\mu\in\sigma(A)\backslash\{\mu^*,\cj{\mu}^*\}$. A left or right eigenvector is a \emph{(strictly) leading eigenvector} if it is associated to a (strictly) leading eigenvalue. 

An $N\times N$ matrix is said to be \emph{in-regular} if it has the all-ones vector $\vo{1}_N$ as a right eigenvector.
We order the elements $\mu_1,\ldots,\mu_N$ of $\sigma(A)$ decreasingly by their real parts, i.e., ${\rm Re}(\mu_j)\geqslant {\rm Re}(\mu_{j+1})$ for all $j\in\{1,\ldots, N-1\}$. Simple conjugate eigenvalues are ordered decreasingly by their imaginary parts. Real repeated eigenvalues are ordered arbitrarily as consecutive elements; for repeated non-real eigenvalues, we write them in conjugate pairs and order each pair by their imaginary parts. If some real and non-real eigenvalues have equal real parts, real ones appear first, followed by the non-real ones which are ordered by their imaginary parts. If $A$ has a real strictly leading eigenvalue, then we denote it as $\mu_1\in\sigma(A)$; if $A$ has a conjugate couple of non-real strictly leading eigenvalues, then we denote them as $\mu_1$ and $\mu_2=\cj{\mu}_1$, such that $\mathrm{Im}(\mu_1)>0$. 

A $2N$-tuple $\vo{z}$ may be denoted in {\it block-wise notation} by $\vo{z}=(\vo{x}^t\vert\vo{y}^t)^t$, where $\vo{x}$ and $\vo{y}$ are $N$-tuples. The block-wise notation for a $2N\times 2N$ matrix $M$ is
$M = \left(\begin{array}{c|c}
    A & B  \\
    \hline
    C & D
\end{array}\right)$,
where $A$, $B$, $C$, and $D$ are $N\times N$ matrices. Operations between block-wise defined matrices and vectors are such that
$\left(\begin{array}{c|c}
    A & B  \\
    \hline
    C & D
\end{array}\right)\left(\begin{array}{c}
    x  \\
    \hline
    y
\end{array}\right)=\left(\begin{array}{c}
    Ax + By  \\
    \hline
    Cx + Dy
\end{array}\right)$.

A {\it weighted and signed digraph}, or {\it network}, is defined as a triplet $\mathscr{G}=(V,E,A)$, where $V=\{v_1,\ldots,v_N\}$ is the set of vertices, $E$ is the set of edges connecting the elements in $V$, and $A$ is the weighted adjacency matrix whose entry $A_{jk}$ determines the weight and the sign of the connection from the $k$th node to the $j$th node. A digraph is said to be \emph{strongly connected} if for any two vertices there exists a directed path connecting them, which is equivalent to $A$ being irreducible \cite[Theorem 3.2.1]{brualdi}. A weighted digraph is in-regular if there exists $d\in\mathds{R}$ such that the sum of edge weights into each node is $d$, which is equivalent to $A$ having $\vo{1}_N$ as a right eigenvector.

\section{Mixed-feedback Networks}\label{sec:model}

\begin{figure}
    \centering
    \includegraphics[width=0.48\textwidth]{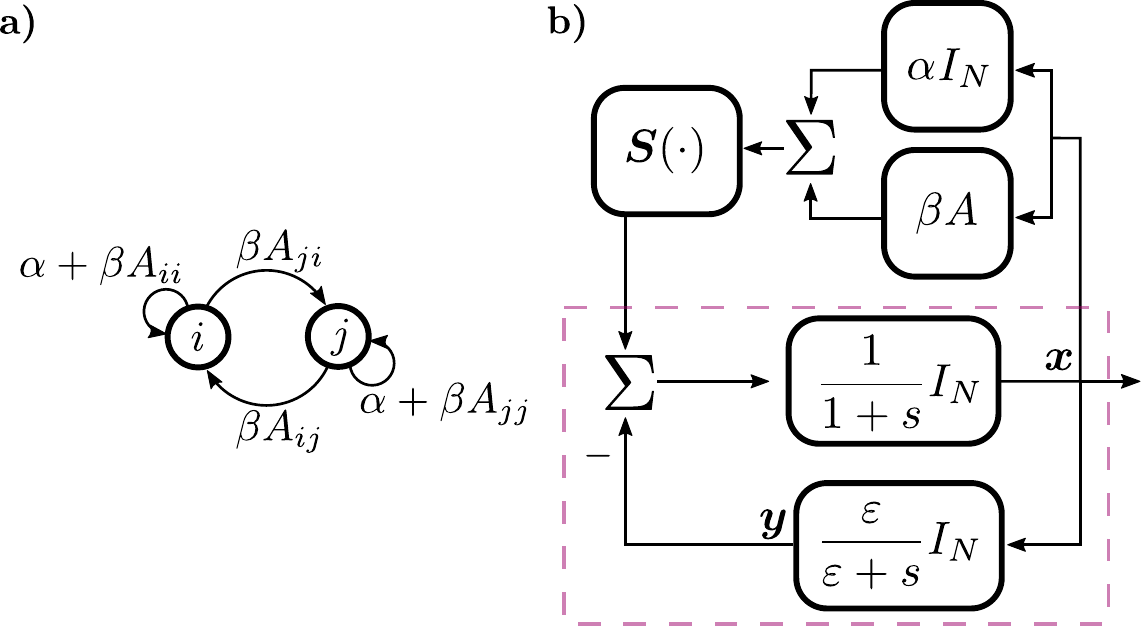}
    \caption{\textbf{a)} Network representation of model~\eqref{eq:ei} for a pair of nodes $i$ and $j$ showing self-loops and inter-node edges. \textbf{b)} Block diagram representation of model~\eqref{eq:ei}. The vector $\vo{x}$ is the system output. The vector $\vo{y}$ is a  lagged version of $\vo{x}$ providing slow-negative feedback on $\vo{x}$. The purple dashed box indicates the node-level linear system component. The blocks outside the dashed blue box are nonlinear, networked, and provide fast  positive feedback on $\vo{x}$.}
    \label{fig:net_2}
\end{figure}

Consider a  $2N$-dimensional dynamical system
\begin{subequations}\label{eq:ei}
    \begin{align}
        \dot{x}_j&=-x_j-y_j+S\left(\alpha x_j+\beta\sum_{k=1}^NA_{jk}x_k\right),\\
        \dot{y}_j&=\varepsilon(x_j-y_j), \quad\quad j=1,\ldots,N,
    \end{align}
\end{subequations}
where $0<\varepsilon\ll1$ is a small positive time constant and $S$ is a \emph{locally odd sigmoid satisfying} (a) $S(0)=0$, (b) $\forall\,x\in\mathds{R}:\,S'(x)>0$, and (c) $\mathrm{argmax}\, S'(x)=0$; assuming $S$ is at least three-times differentiable, these conditions imply $S''(0)=0$ and $S'''(0)\leqslant0$. In the sequel we assume $S'(0)=1$ and $S'''(0)<0$; for simulations, we use $S(\cdot)=\tanh(\cdot)$. Observe the distinct nature of variables $x_j$ and $y_j$: while the former is the {\it output} of node $j$ that is transmitted through the network to other nodes, the latter provides fully local {\it slow negative feedback} on $x_j$. As we will see, the {\it local} slow negative feedback provided by the $y_j$s turns nonlinear network interactions mediated by the $x_j$s into robust and easily controllable collective, i.e., {\it global}, network oscillations.

Matrix $A\in\mathds{R}^{N\times N}$ is the adjacency matrix of a network $\mathscr{G}$ with vertices $V=\{1,\ldots,N\}$. Parameter $\alpha\geqslant 0$ models the average node-level self-loop weight, while parameter $\beta\geqslant 0$ models the average nodal interaction strength. The choice of introducing network-wide parameters $\alpha$ and $\beta$ is tailored to defining a few tuning dials (bifurcation parameters) that control the network behavior, but since no topological conditions are imposed over matrix $A$, the model remains general. Our goal is to study the emergence of network oscillations and their rhythmic profile in model~\eqref{eq:ei} as parameters $\alpha$ and/or $\beta$ are varied. In this sense,~\eqref{eq:ei} can be considered a model of \emph{rhythmogenesis}, i.e. capable of describing the transition from non-rhythmic to rhythmic behavior~\cite{shamir19}.

Figure~\ref{fig:net_2}a illustrates how $\alpha$, $\beta$, and $A$ determine the weights of both inter-node edges and self-loops in a network described by~\eqref{eq:ei}. Figure~\ref{fig:net_2}b provides an equivalent block diagram representation of~\eqref{eq:ei} as the MIMO feedback interconnection of $N$ identical linear systems (blue dashed box) with a vector saturation nonlinearity. Each linear system is the {\it negative feedback} interconnection of a fast lag ($x_j$) and a {\it slow} lag ($y_j$). The parameter $0<\varepsilon\ll 1$ sets the timescale of the slow lag. The saturation nonlinearity aggregates the effects of both local nonlinear interactions (self-loops in Figure~\ref{fig:net_2}a), with gain $\alpha I_N$, and networked interactions (inter-node edges in Figure~\ref{fig:net_2}a), with gain $\beta A$, among the $N$ linear systems. Nonlinear local interactions provide fast {\it positive feedback} on each $x_j$ with gain $\alpha$. In other words, the local node-level dynamics
\begin{equation}\label{eq:ei_indi}
    \begin{split}
        \dot{x}_j&=-x_j-y_j+S(\alpha x_j),\\
        \dot{y}_j&=\varepsilon(x_j-y_j),        
    \end{split}
\end{equation}
obtained by setting $A=O_N$ in~\eqref{eq:ei}, are a simple realization of a fast-positive, slow-negative, mixed-feedback system. Nonlinear (saturated) network interactions describe recurrent interconnections among these $N$ mixed-feedback systems.

The specific form~\eqref{eq:ei_indi} of the node-level mixed-feedback dynamics is mathematically tractable because the only nonlinear term is a locally-odd saturation nonlinearity that localizes in range the node-level positive feedback. An equivalent realization was thoroughly analyzed in~\cite{franci2019sensitivity,ribar2019neuromodulation}. As shown in~\cite{ribar2019neuromodulation}, this realization can also be implemented in standard analog neuromorphic electronics. Finally, the mixed-feedback interconnection of linear systems and saturation nonlinearities is at the foundation of the mathematical tractability of fast-and-flexible decision-making models~\cite{leonard2024fast}.

Because $0<\varepsilon\ll1$, model~\eqref{eq:ei} is fast-slow or {\it singularly perturbed}. Our analysis does not explicitly use singular perturbation methods~\cite{kokotovic1999singular,teel2003unified,arnold1995geometric} because, given the non-linear, multi-equilibrium nature of the fast dynamics~(\ref{eq:ei}a), the analyses of both the layer and reduced problems associated to~\eqref{eq:ei} are, in general, intractable. The timescale separation imposed by a sufficiently small $\varepsilon$ is nonetheless key for the tractability~\eqref{eq:ei} and, specifically, to ensure that its nonlinear behavior can be shaped through network design. In particular, all the results in Sections~\ref{sec:dom} require a sufficiently small $\varepsilon$ to hold.

Model~\eqref{eq:ei} can be equivalently stated as $\dot{\vo{z}}=\vo{f}(\vo{z};\alpha,\beta)$,
where vector field $\vo{f}=(f_1,\ldots,f_{2N}):\mathds{R}^{2N}\to\mathds{R}^{2N}$ is defined entry-wise by
\begin{align*}
    f_j(\vo{z};\alpha,\beta)&=-\vo{z}_j-\vo{z}_{j+N}+S\left(\alpha \vo{z}_j+\beta\sum_{k=1}^NA_{jk}\vo{z}_k\right)\\
    f_{j+N}(\vo{z};\alpha,\beta)&=\varepsilon(\vo{z}_j-\vo{z}_{j+N}), \quad\quad j=1,\ldots,N.
\end{align*}

Since $S$ is a locally odd sigmoid, it follows that $\vo{z}_0=(\vo{0}_N^t\vert\vo{0}_N^t)^t$ is always an equilibrium of model~\eqref{eq:ei}. Evaluating the Jacobian matrix at this equilibrium readily yields the following block-wise expression for the $2N\times 2N$ matrix $J_0=J_{\alpha,\beta,A,\varepsilon}(\vo{0}_N,\vo{0}_N)$,
\begin{equation}\label{eq:jacob}
    J_0:=\left(\begin{array}{c|c}
        (\alpha-1)I_N+\beta A & -I_N  \\
        \hline
        \varepsilon I_N & -\varepsilon I_N 
    \end{array}\right).
\end{equation}

\section{Control of rhythmic networks: problem formulation and results overview}\label{sec:relevance}

In this section we introduce the notion of a network rhythmic profile  and its graphical representation. We then formulate the main control problem attacked in this paper and use the graphical representation of network rhythmic profile to easily visualize the performance of a tentative solution to this problem. We then overview some key aspects of the proposed solution in this paper and provide a summary of the main technical results presented in the next sections.

\subsection{Rhythmic profiles}

We say that a network of coupled oscillators is \emph{rhythmic} if its trajectories (at least for some initial conditions) converge to a limit cycle or, in other words, if all of its nodes exhibit asymptotically periodic oscillations with the same period $T>0$. The {\it rhythmic profile} of the network is then defined by the amplitudes and phase differences of node oscillations. To formalize these ideas, we first introduce the notion of {\it oscillating function}, as a generalization of simple periodic functions such as $\sin(\cdot)$ and $\cos(\cdot)$.

\begin{mydef}[{\bf Oscillating function}]
    A function $r:\mathds{R}\to\mathds{R}$ is called \emph{oscillating} if it is $T$-periodic, with $T>0$, and there exists $0<T_{1/2}<T$ such that (1) $r(0)=r(T_{1/2})=r(T)=0$, (2) $r(t)>0$ for $t\in(0,T_{1/2})$, (3) $r(t)<0$ for $t\in(T_{1/2},T)$, and (4) its range is normalized such that $\max\{r(t):\,t\in[0,T]\}-\min\{r(t):\,t\in[0,T]\}=2$.
\end{mydef}

\begin{mydef}[\textbf{Rhythmic network and rhythmic profile}]  \label{def:rp}
Consider a network $\mathscr{G}$ with vertices $V=\{1,\ldots,N\}$. Suppose that the state of each vertex is described by a state variable $\vo{x}_{j}\in\mathds{R}^n$ and that the network state $\vo{X}=(\vo{x}_1^t\vert\cdots\vert\vo{x}_N^t)^t$ evolves according to $\dot{\vo{X}}=\vo{f}(\vo{X})$, where $\vo{f}:\mathds{R}^{Nn}\to\mathds{R}^{Nn}$ is smooth. Let $x_{j1}=(\vo{x}_j)_1$ be the output of node $j$. We say that the network $\mathscr{G}$ is \emph{rhythmic} if there exist $N$ oscillating functions $r_1,\ldots,r_{N}:\mathds{R}\to\mathds{R}$, $N$ \emph{amplitudes} $\sigma_1,\ldots,\sigma_{N}\in\mathds{R}$, $N$ \emph{phases} $\varphi_1,\ldots,\varphi_{N}\in [0,2\pi)$, and an open set $U\subset\mathds{R}^{Nn}$, such that the solution $\vo{X}(t)$ to $\dot{\vo{X}}=\vo{f}(\vo{X})$, $\vo{X}(0)=\vo{X}_0$, satisfies $\lim_{t\to\infty}\left|x_{j1}(t)-\sigma_j r_j\left(t+\dfrac{T\varphi_j}{2\pi}\right)\right| =0$
for all $j=1,\ldots,N$, whenever $\vo{X}_0\in U$. The \emph{rhythmic profile} of $\mathscr{G}$ is the $N$-tuple $(\sigma_1e^{i\varphi_1},\ldots,\sigma_Ne^{i\varphi_N})\in\mathds{C}^N$.
\end{mydef}

We can represent the rhythmic profile of a network on the complex unitary disc $\mathds{D}^2=\{z\in\mathds{C}:\,\abs{z}\leqslant1\}$ by expressing all amplitudes and phases relative to the amplitude and phase of the oscillator with largest amplitude.

\begin{mydef}[\textbf{Relative rhythmic profile}] \label{def:rel_rp}
Consider a rhythmic network $\mathscr{G}$ and suppose that $\sigma_1>0$ and $\sigma_1\geqslant \sigma_{j}$ for all $j\neq 1$. Then the \emph{relative rhythmic profile} of $\mathscr{G}$ is defined as the $N$-tuple $(1,\rho_2e^{i\theta_2},\ldots,\rho_Ne^{i\theta_N})\in(\mathds{D}^2)^N$, where $\rho_j=\tfrac{\sigma_j}{\sigma_1}$ and $\theta_j=\varphi_j-\varphi_1$.
\end{mydef}

\begin{figure}
    \centering
    \includegraphics[width=0.48\textwidth]{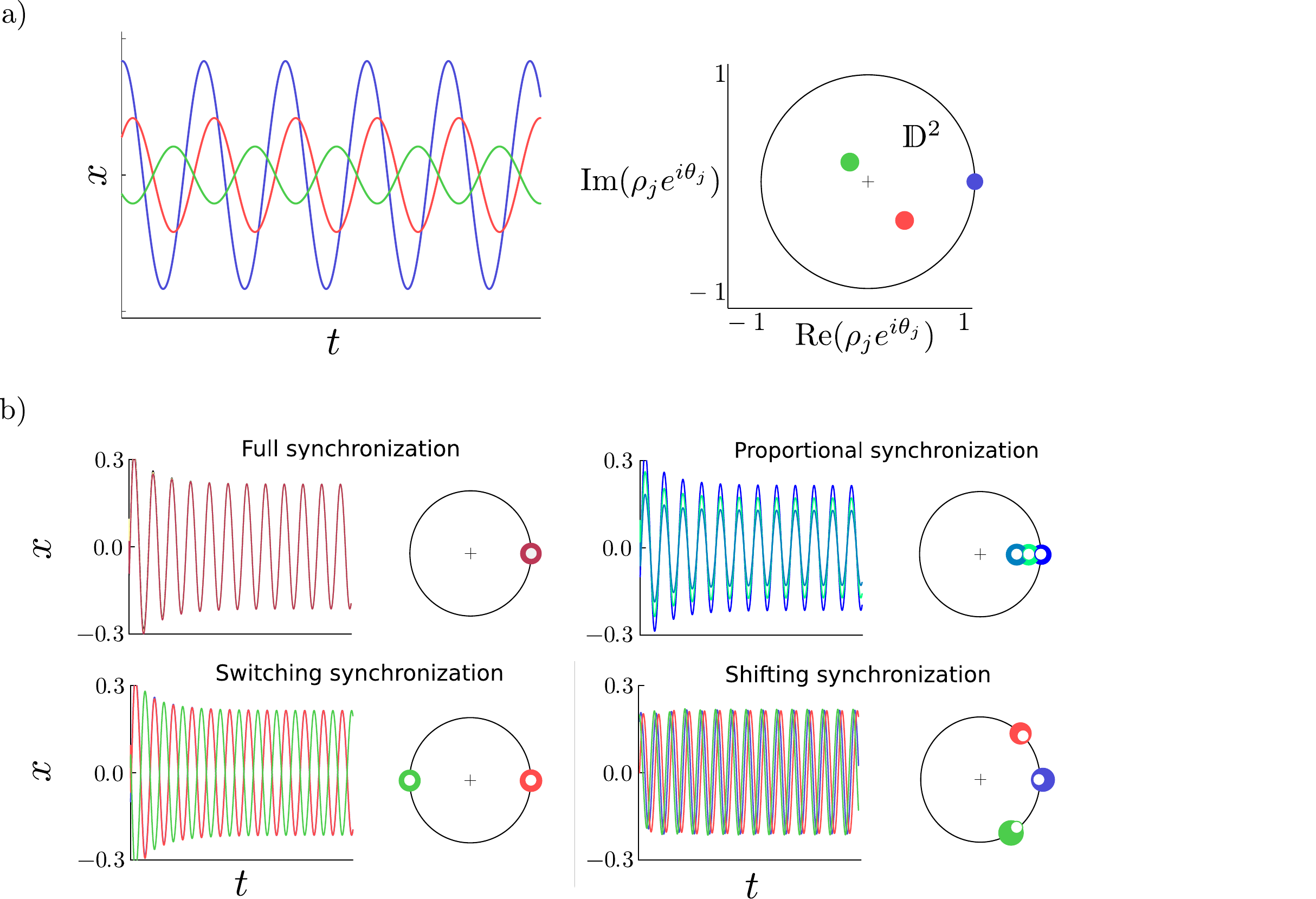}
    \caption{\textbf{a)} Output of a rhythmic network and its amplitude and phase relationships geometrically represented through its relative rhythmic profile in $\mathds{D}^2$. Oscillators are color-coded in both representations. \textbf{b)} Common rhythmic profiles and their prediction. In each panel, the activity pattern of the nodes in a rhythmic network is accompanied by the color-coded geometrical representation of its relative rhythmic profile. Solid and hollow points represent predicted (by our theory) and observed rhythmic profiles, respectively. See text for details.}
    \label{fig:profiles}
\end{figure}

The relative rhythmic profile can be represented geometrically as the set $\{1,\rho_2e^{i\theta_2},\ldots,\rho_Ne^{i\theta_N}\}$ of $N$ points in the complex unitary disc $\mathds{D}^2$. Figure~\ref{fig:profiles}a shows the relative rhythmic profile of networks with $N=3$ oscillators. In what follows, when depicting rhythmic profiles, including Figure~\ref{fig:profiles}b, we will omit the real and imaginary axes, and solely represent the elements of the profile in $\mathds{D}^2$. The relative rhythmic profile allows for a concise classification of common rhythmic behaviors.

\begin{mydef}[\textbf{Common rhythmic profiles}] \label{def:rp_types}
Let $\mathscr{G}$ be a rhythmic network and $(1,\rho_2e^{i\theta_2},\ldots,\rho_Ne^{i\theta_N})$ denote its relative rhythmic profile. Suppose $r_1=\cdots=r_N$, i.e., each node is oscillating with the same periodic wave form but possibly different amplitude and phases. The oscillators are then said to be \emph{phase-locked}. Moreover, if one of the following holds for all $j\in\{1,\ldots,N$\}, then the network is said to be:
\begin{itemize}
    \item \emph{Fully synchronized}, if $\rho_j=1$ and $\theta_j=0\,\mathrm{mod}\,2\pi$.
    \item  \emph{Proportionally synchronized}, if $\theta_j=0,\mathrm{mod}\,2\pi$.
    \item \emph{Switching synchronized} if $\rho_j=1$ and $\theta_j\in\{0,\pi\}\,\mathrm{mod}\,2\pi$.
    \item \emph{Shifting synchronized} if $\rho_j=1$.
\end{itemize}
\end{mydef}

These definitions generalize some rhythmic phenomena found in the literature. For instance, proportional and switching synchronization are the network-level versions of in-phase and anti-phase synchronization~\cite{krishnagopal16, williams13}, respectively. Shifting synchronization with homogeneous phase differences describes a behaviour similar to \emph{travelling waves} \cite{ramamoorthy22}, lag synchronization~\cite{rosenblum97}, and \emph{splay-phase behaviour}~\cite{nichols92, zillmer07}, under constant or non-constant phase differences~\cite{zou09}.
Figure~\ref{fig:profiles}b presents the network oscillations and the resulting geometric representations of the rhythmic profiles introduced in Definition \ref{def:rp_types}. Observe that all of these network oscillations are phase-locked (as per Definition~\ref{def:rp}) and asymptotically periodic (as per Definition~\ref{def:rp_types}).

\subsection{Predicted vs observed relative rhythmic profiles}

Solid dots in the geometrical representations of rhythmic profiles, such as those found in Figure~\ref{fig:profiles}b, denote the {\it predictions} obtained with the techniques and results developed in this paper, which solely use the spectral properties of the network adjacency matrix. Hollow dots represent the measured relative rhythmic profiles of network oscillations; as seen in Figure~\ref{fig:profiles}, our predictions match the observed behavior with very small (often zero) errors.

\begin{restatable}{myrem}{rem:measure}
    Using the results in Section~\ref{sec:dom}, it is possible to show that the prediction error of the proposed collective rhythm design solution is small, in the sense that the Hausdorff distance between the observed and predicted rhythmic profile close to the Hopf bifurcation at which the network rhythms emerge (Theorems~\ref{theo:rhythm} and~\ref{theo:rhythm_beta}) is $O(\varepsilon^{0.5})$. 
\end{restatable}

\subsection{Constructive rhythm control, network structure, and bifurcations}

\begin{figure}
    \centering
    \includegraphics[width=0.48\textwidth]{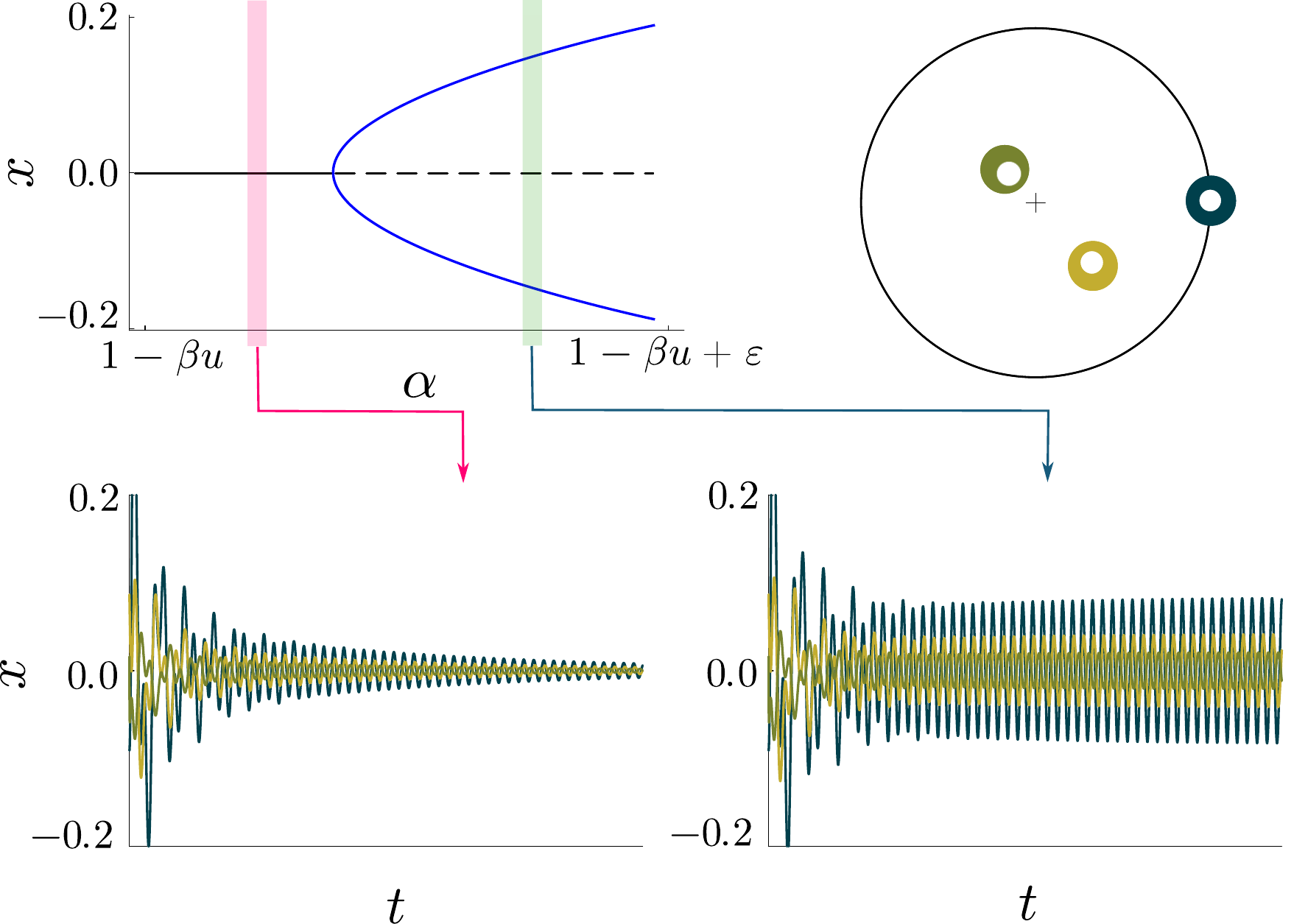}
    \caption{A Hopf bifurcation triggers the transition from damped to sustained oscillations in model~\eqref{eq:ei}. \emph{Top left:} bifurcation diagram of the Hopf bifurcation. Black continuous (dashed) lines represent branches of stable (unstable) equilibria. Blue continuous lines represent branches of stable limit cycles. The loss of stability of the equilibrium point and the appearance of a stable limit cycle are the hallmark of a supercritical Hopf bifurcation. \emph{Top right:} rhythmic profile of an $N=3$ network; as in Fig.~\ref{fig:profiles}, full and hollow nodes represent predicted and measured rhythmic profiles, respectively. \emph{Bottom:} pre- and post-bifurcation evolution for variables $x_j$ in the same network as the top plots.}
    \label{fig:bifcontrast}
\end{figure}

How network structure determines a rhythmic profile is a core question of this work. We will see how, in networks of fast-slow coupled oscillators~\eqref{eq:ei}, our techniques allow us not only to precisely {\it analyze} (and thus predict) the emergent rhythmic profile, but also to {\it design} networks to exhibit a desired rhythmic profile. Indeed, we can formulate the two following problems within the context of rhythm control and prediction:
\begin{itemize}
    \item \emph{Direct problem:} Can we successfully predict the activity pattern of a given rhythmic network by looking at its adjacency matrix?
    \item \emph{Inverse problem:} Can we construct an adjacency matrix such that a particular rhythmic profile is attained?
\end{itemize}
We can solve both problems simultaneously because of the constructive nature of our methods. We develop such a constructive methodology by relying on the Hopf bifurcation theorem, as sketched by Figure~\ref{fig:bifcontrast}. To develop some intuition on the importance of the Hopf bifurcation for model~\eqref{eq:ei}, let us consider local node-level dynamics~\eqref{eq:ei_indi}. The Jacobian matrix of this system at equilibrium $(x_0,y_0)=(0,0)$ is
$J_0=\left(\begin{array}{cc}
   \alpha-1  & -1 \\
   \varepsilon  & -\varepsilon 
\end{array}\right)$,
whose eigenvalues are
$\lambda_{1,2}=\tfrac{\alpha-1-\varepsilon\pm\sqrt{(\alpha+\varepsilon-1)^2-4\varepsilon(2-\alpha)}}{2}$,
from whence it follows that, whenever $\varepsilon\in(0,1)$, subsystem \eqref{eq:ei_indi} undergoes a Hopf bifurcation with bifurcation parameter $\alpha$ at a critical value $\alpha^*=1+\varepsilon$. Close to bifurcation, the period is $T=2\pi\sqrt{\varepsilon(1-\varepsilon)}^{-1}$.

The case of uncoupled oscillators, although tractable, is not satisfactory from the engineering perspective because phase differences will depend on initial conditions and will not be robustly maintained in the presence of disturbances. In this paper we will develop methods to predict (or design) the emergence of a {\it networked} Hopf bifurcation that will lead to robust, initial condition-independent rhythmic profiles in a way that is fully determined by network structure. In particular, we will show how the leading eigenstructure of the network adjacency matrix fully determines the emergent rhythmic profile.

\subsection{Summary of results}

In the following sections we present the results needed to prove the rich variety of rhythmic profiles exhibited by model~\eqref{eq:ei} for certain, provable parameter combinations, as well as the possibility of  shaping these rhythmic profiles through network design. These results illustrate different specific aspects of the general fact that {\it mixed-feedback systems generate mathematically tractable nonlinear behaviors.}

Section~\ref{sec: fast dom} states and interprets the main ``dominance'' assumption. The only condition required by this assumption is an elementary algebraic one, namely, that the adjacency matrix $A$ has simple (either real or complex) \emph{leading eigenvalues}. We discuss the fundamental system-theoretic implications of this assumption for the dominance properties (in the sense of~\cite{forni2018differential}) of the fast dynamics of~\eqref{eq:ei}, as well as geometric and graph-theoretical criteria to inspect it or design it.

Section~\ref{sec:struc} exploits the block matrix form of the model Jacobian~\eqref{eq:jacob}, enforced by the mixed-feedback structure of~\eqref{eq:ei}, to rigorously characterize the relationship between the spectral properties of the adjacency matrix $A$ and of the Jacobian $J_0$ of~\eqref{eq:ei} at the origin. These results thus show that, in networks of mixed-feedback systems described by~\eqref{eq:ei}, the network topology determines the network linearized dynamical behavior in a transparent and tractable way.

Section~\ref{sec:dom} builds upon these results to show that, if the dominance assumption of~Section~\ref{sec: fast dom} is satisfied and if $\varepsilon$ is sufficiently small (i.e., the mixed-feedback timescale separation is large enough), then there always exist critical values for the self-loop strength $\alpha$ or the interconnection strength $\beta$ in~\eqref{eq:ei}, such that the Jacobian $J_0$ has a pair of purely imaginary, simple, leading eigenvalues inherited by the (either real or complex) leading eigenstructure of $A$. It further characterizes the sensitivity of the leading eigenvalues of $J_0$ to changes in $\alpha$ and $\beta$. The results presented in Section~\ref{sec:struc} thus show that, in networks of mixed-feedback systems described by~\eqref{eq:ei}, dominance properties of the full fast-slow system can easily be designed by suitably designing the network topology.

Finally, Section~\ref{sec:synchro} states the existence of a Hopf bifurcation (Theorems~\ref{theo:rhythm} and~\ref{theo:rhythm_beta}) through both model parameters $\alpha$ and $\beta$ as a consequence of the results presented in previous sections. Both bifurcation parameters are equivalent in the sense that the dynamical behaviors emerging at bifurcation are the same. The key difference is that changing parameter $\alpha$ models a local, node-level, modulation of positive feedback strength, whereas changing parameter $\beta$ models a distributed, network-level, modulation of positive feedback strength. Which parameter is relevant is therefore a matter of specific applications or interpretations of the model. Furthermore, a linear approximation to the system exactly at bifurcation reveals that the rhythmic pattern of the network is determined by the leading eigenvector of adjacency matrix $A$ (see Proposition~\ref{prop:synchro}). This result is then applied to the construction of networks with specific rhythmic patterns in two particular cases under System~\eqref{eq:ei} in Section~\ref{sec:build}.

\section{A fast dominance assumption}\label{sec: fast dom}

The key ingredients of our approach are the fast-slow nature of the oscillators and the following assumption.

\begin{myas}\label{as:main leading}
    The adjacency matrix $A$ in model~\eqref{eq:ei} has a strictly leading real eigenvalue $\mu_1>0$ or strictly leading complex conjugate eigenvalues $\mu_1,\mu_2=\cj{\mu}_1$, $\mathrm{Re}(\mu_1)>0$.
\end{myas}

Assumption~\ref{as:main leading} is key to our approach because it implies that the linearization at the origin of the fast dynamics~(\ref{eq:ei}a) of model~\eqref{eq:ei} possesses, for suitable $\alpha$ and $\beta$, low-dimensional {\it dominant dynamics}~\cite{forni2018differential}. Indeed, if $A$ has a strictly leading eigenvalue $\mu_1$, then the Jacobian $J_0^f=(\alpha-1)I_N+\beta A$ of the fast subsystem, i.e., the upper left block in~\eqref{eq:jacob}, has also a strictly leading eigenvalue $\alpha-1+\beta\mu_1$. Hence, for $\alpha= 1 - \beta\mu_1$ the strictly leading eigenvalue of $J^f$ is purely imaginary, while all non leading eigenvalues have negative real part. Invoking~\cite[Proposition~1]{forni2018differential}, the linearization of fast dynamics~(\ref{eq:ei}a) is 1-dominant, if $\mu_1$ is real, or 2-dominant, if $\mu_1$ is non-real.

Dominance implies that, close to the origin, the $N$-dimensional fast dynamics~(\ref{eq:ei}a) effectively behave as low-dimensional dynamics, in the sense that they possess an $N-1$ (1-dominance case) or $N-2$ (2-dominance case) ``fast'' or ``non-dominant'' subspace along which the trajectories converge exponentially to zero. All the interesting nonlinear dynamical behaviors are therefore restricted to the 1- or 2-dimensional ``slow'' or ``dominant'' complementary subspace, which makes the analysis tractable.

When either $\alpha$ or $\beta$ are increased and the strictly leading eigenvalue $\alpha-1+\beta\mu_1$ crosses the imaginary axis, the fast dynamics become linearly unstable: a bifurcation happens inside the dominant subspace $\mathcal H$, at which intrinsically nonlinear (but low-dimensional) dynamical behaviors, like multi-stability or limit cycle oscillations, can emerge.

In the remainder of the paper we will show that fast dominant dynamics are inherited by the full fast-slow system~\eqref{eq:ei}. Furthermore, the loss of stability of the dominant dynamics of model~\eqref{eq:ei} necessarily leads to limit cycle oscillations through a Hopf bifurcation. The dominant eigenstructure of $A$ fully determines the critical parameter values at which the bifurcation happens as well as the rhythmic profile associated to the emerging limit cycles, thus providing a constructive methodology for network rhythm control.

\subsection{Sufficient conditions for fast dominance}

A well known sufficient condition for the existence of a strictly leading real eigenvalue $\mu_1=\rho(A)>0$ with a positive eigenvector $\vo{w}_1$, is the Perron-Frobenius theorem~\cite[Theorem 8.4.4]{horn}, which applies to matrices with non-negative entries, and which was generalized to matrices with mixed-sign entries in~\cite{noutsos2006perron}. A second generalization stems from the notion of (structurally) balanced networks~\cite{harary1953notion,sontag2007monotone,altafini2012dynamics}. The adjacency matrix associated to a structurally balanced network possesses a strictly leading real positive eigenvalue, but this eigenvalue is neither guaranteed to be the spectral radius of the matrix nor is the corresponding eigenvector guaranteed to be positive. A summary of conditions under which a graph is defined by an adjacency matrix with a strictly leading real positive eigenvalue can be found in~\cite[Lemma 2.2]{bizyaeva2023multi}. To the best of the authors' knowledge, no general conditions were ever proved for the existence of a strictly leading complex conjugate eigenvalue pair.

\section{Characterization of the eigenstructure of $J_0$ in terms of the eigenstructure of $A$}
\label{sec:struc}

We start by showing how spectral properties of Jacobian matrix $J_0$ in Equation~\eqref{eq:jacob} are determined by those of adjacency matrix $A$, and vice versa. Namely, we derive formulae to compute the $2N$ $J_0$-eigenvalues in terms of the $N$ $A$-eigenvalues. We also show that the $J_0$-eigenvectors, both left and right, inherit the structure of corresponding $A$-eigenvectors. The technical proofs of the results in this section are provided in the online version \cite{juarezalvarez2024analysis}.

\begin{restatable}{mylem}{rquadfirst}
\label{lem:rquad}
    $\mu\in\sigma(A)$ if and only if there exists $\lambda\in\sigma(J_0)$ such that
\begin{equation}\label{eq:rquad}
  \lambda^2+(1+\varepsilon-\alpha-\beta\mu)\lambda+\varepsilon(2-\alpha-\beta\mu)=0,  
\end{equation}
    or, equivalently,
\begin{equation}\label{eq:rquad_eq}
    \mu=\dfrac{1-\alpha+\lambda+\tfrac{\varepsilon}{\varepsilon+\lambda}}{\beta}.
\end{equation}
    Moreover, for any $\mu\in\sigma(A)$, if $\vo{w}_x\in\mathds{C}^N$ is an associated right $A$-eigenvector, then $\vo{w}=(\vo{w}_x^t\vert\tfrac{\varepsilon}{\varepsilon+\lambda}\vo{w}_x^t)^t\in\mathds{C}^{2N}$ is a right $J_0$-eigenvector associated to $\lambda\in\sigma(J_0)$ satisfying condition~\eqref{eq:rquad}. Conversely, for any $\lambda\in\sigma(J_0)$, if $\vo{w}=(\vo{w}_x^t\vert\vo{w}_y^t)^t\in\mathds{C}^{2N}$ is the associated right $J_0$-eigenvector then necessarily 
\begin{equation}\label{eq:reigvec}
A\vo{w}_x=\mu\vo{w}_x,\quad\quad \vo{w}_y=\tfrac{\varepsilon}{\varepsilon+\lambda}\vo{w}_x,
\end{equation}
where $\mu\in\sigma(A)$ satisfies condition~\eqref{eq:rquad_eq}.    
\end{restatable}

\begin{restatable}{mylem}{lquadfirst}\label{lem:lquad}
   For any $\mu\in\sigma(A)$, if $\vo{v}_x\in\mathds{C}^N$ is an associated left $A$-eigenvector then $\vo{v}=(\vo{v}_x^t\vert\tfrac{-1}{\varepsilon+\lambda}\vo{v}_x^t)^t\in\mathds{C}^{2N}$ is a left $J_0$-eigenvector associated to $\lambda\in\sigma(J_0)$ satisfying condition \eqref{eq:rquad}. Conversely, for any $\lambda\in\sigma(J_0)$, if $\vo{v}=(\vo{v}_x^t\vert\vo{v}_y^t)^t\in\mathds{C}^{2N}$ is the associated left $J_0$-eigenvector then necessarily
\begin{equation}\label{eq:leigvec}
\vo{v}_x^t A=\mu\vo{v}_x^t,\quad\quad \vo{v}_y=\tfrac{-1}{\varepsilon+\lambda}\vo{v}_x,
\end{equation}
where $\mu\in\sigma(A)$ satisfies condition \eqref{eq:rquad_eq}.
\end{restatable}

Given $\mu=u+iv\in\sigma(A)$, we denote the two $J_0$-eigenvalues associated to $\mu$ guaranteed by Lemma \ref{lem:rquad} as $\nmax{\mu}(\alpha,\beta,\varepsilon)=\fmax{\mu}(\alpha,\beta,\varepsilon)+i\pmax{\mu}(\alpha,\beta,\varepsilon)$
and $\nmin{\mu}(\alpha,\beta,\varepsilon)=\fmin{\mu}(\alpha,\beta,\varepsilon)-i\pmin{\mu}(\alpha,\beta,\varepsilon)$, where the exact expressions for real functions $\fmax{\mu}$, $\pmax{\mu}$, $\fmin{\mu}$, and $\pmin{\mu}$ are obtained through the formulae for the principal roots of complex numbers (see Eq.~\eqref{eq:lead2} of Appendix~\ref{app:results} in the online version~\cite{juarezalvarez2024analysis}). Per our definition, the condition $\fmin{\mu}\leqslant\fmax{\mu}$ is verified.

\begin{mydef}[\textbf{Associated eigenvalues}]\label{def:assoc}
    Let $\mu\in\sigma(A)$ be an $A$-eigenvalue. Then for any $\alpha\in\mathds{R}$, $\beta\in\mathds{R}$, and $\varepsilon\geqslant0$, the two $J_0$-eigenvalues $\nmax{\mu}(\alpha,\beta,\varepsilon)$, $\nmin{\mu}(\alpha,\beta,\varepsilon)$ are called the $J_0$-\emph{eigenvalues associated} to $\mu$.
\end{mydef}

To simplify the notation, in the sequel we drop the dependence of functions $\nu^{\pm}_{\mu}$ on parameters $\alpha$, $\beta$, and $\varepsilon$. We use $\nu_{j}^\pm$ as a shorthand for $\nu_{\mu_j}^\pm$. Recall that the elements $\mu_1,\ldots,\mu_N$ of $\sigma(A)$ are ordered decreasingly by their real parts (see Section \ref{sec:math}). The following lemma provides conjugation relationships between $\nMIN_j$ and $\nMAX_j$. 

\begin{restatable}{mylem}{conjfirst}
\label{lem:conj}
    Let $\mu_j\in\sigma(A)$, for $j\in\{1,\ldots,N\}$, and $\nmin{j}$, $\nmax{j}$ denote its associated $J_0$-eigenvalues as in Definition \ref{def:assoc}. Then the following hold for small enough values of $\varepsilon>0$.\\ 
    {\it a)} If $\mu_j\in\mathds{R}$ and $\{\nmin{j},\nmax{j}\}\subset\mathds{R}$, then $\nmin{j}<\nmax{j}$.\\
    {\it b)} If $\mu_j\in\mathds{R}$ and $\{\nmin{j},\nmax{j}\}\subset\mathds{C}\backslash\mathds{R}$, then $\nmin{j}=\cjnmax{j}$.\\
    {\it c)} If $\mu_j\in\mathds{C}\backslash\mathds{R}$, $j<N$, then $\nmax{j+1}=\cj{\nu}_j^+$ and $\nmin{j+1}=\cjnmin{j}.$
\end{restatable}

\section{Conditions for purely imaginary $J_0$ leading eigenvalues and resulting dominant dynamics}
\label{sec:dom}

In this section we use Assumption~\ref{as:main leading} to characterize the leading eigenstructure of $J_0$, provided knowledge of the leading eigenstructure of $A$. In particular, we obtain conditions on $\alpha$, $\beta$, $\varepsilon$ under which the $J_0$-eigenvalues $\nu_{1}^\pm$ associated to a strictly leading real $A$-eigenvalue $\mu_1$ or the $J_0$-eigenvalues $\nu_{1}^+,\nu_2^+$ associated to a strictly leading complex conjugate $A$-eigenvalues $\mu_1$, $\mu_2=\cj{\mu}_1$ have zero real part, while all the other $J_0$-eigenvalues have negative real part. In other words, we show how the leading eigenstructure of $A$ and the resulting fast dominant dynamics (Section~\ref{sec: fast dom}) map to the leading eigenstructure of $J_0$ and to a (linearized) 2-dominant dynamics for the full fast-slow dynamics~\eqref{eq:ei}. Finally, we characterize how the parameter variations affect the leading eigenvalues of $J_0$, which will be instrumental for bifurcation analysis. Proofs of the results in this section are provided in the online version \cite{juarezalvarez2024analysis}.

\begin{restatable}{mylem}{rexistfirst}
\label{lem:rexist}
    Let the strictly leading $A$-eigenvalue $\mu_1$ be real. Let also
    \begin{equation}\label{eq:acrit}
        \acrit{\beta,1}(\varepsilon)=1+\varepsilon-\beta\mu_1.
    \end{equation}
    Then for small enough $\varepsilon\geqslant0$, $\beta\in(0,\tfrac{1}{\mu_1})$, and $\alpha=\acrit{\beta,1}(\varepsilon)$, the associated $J_0$-eigenvalues $\nmin{1}$ and $\nmax{1}$ are given by
    \begin{equation}\label{eq:rmodul}
     \nu^{\pm}_{1}=\pm i\sqrt{\varepsilon(1-\varepsilon)}.
    \end{equation}
    Thus, $\lim_{\varepsilon\to0}\abs{\nu_1^{\pm}}=0$.
    Moreover, $\tfrac{\partial\Phi^{\pm}_{1}}{\partial\alpha}(\acrit{\beta,1}(\varepsilon),\beta,\varepsilon)>0$.
\end{restatable}

\begin{restatable}{mylem}{rexistbetafirst}
\label{lem:rexist_beta}
    Let the strictly leading $A$-eigenvalue $\mu_1$ be real. Let also
    \begin{equation}\label{eq:bcrit}
    \bcrit{\alpha,1}(\varepsilon)=\dfrac{1+\varepsilon-\alpha}{\mu_1},
    \end{equation}
    Then for small enough values of $\varepsilon\geqslant0$, $\alpha\in(0,1)$, and $\beta=\bcrit{\alpha,1}(\varepsilon)$, the associated $J_0$-eigenvalues $\nmin{1}$ and $\nmax{1}$ are given by~\eqref{eq:rmodul}. Moreover, $\tfrac{\partial\Phi^{\pm}_{1}}{\partial\beta}(\alpha,\bcrit{\alpha,1}(\varepsilon),\varepsilon)>0$.
\end{restatable}

\begin{restatable}{mylem}{cexistfirst}
\label{lem:cexist}
    Let the strictly leading $A$-eigenvalue $\mu_1=u+iv$ be non-real with positive real part. Then for small enough values of $\varepsilon\geqslant0$ and $\beta\in(0,\tfrac{1}{\mathrm{Re}(\mu_1)})$ there exists a differentiable function $\acrit{\beta,1}(\varepsilon)$, satisfying $\acrit{\beta,1}(0)=1-\beta u$ such that, for $\alpha=\acrit{\beta,1}(\varepsilon)$, the associated $J_0$-eigenvalues $\nmax{1}$ and $\nmax{2}$ are
    given by
    \begin{equation}\label{eq:cmodul}
    \nu_{1,2}^+=\pm i\left(\dfrac{\beta v}{2}+\dfrac{1}{2}\sqrt{\beta^2v^2+4\varepsilon(2-\alpha-\beta u)}\right),
\end{equation}
    and, in particular, $\lim_{\varepsilon\to0}\abs{\nu_{1,2}^+}=\beta \mathrm{Im}(\mu_1)$. Furthermore,
    \begin{equation}\label{eq:order_bound}
        1-\beta \mathrm{Re}(\mu_1)-\varepsilon<\acrit{\beta,1}(\varepsilon)<1-\beta \mathrm{Re}(\mu_1)+\varepsilon,
    \end{equation}
    \begin{equation}\label{eq:order}
        \acrit{\beta,1}(\varepsilon)=1-\beta\mathrm{Re}(\mu_1)+O(\varepsilon^2),
    \end{equation}
    and $\tfrac{\partial\fmax{1,2}}{\partial\alpha}(\acrit{\beta,1}(\varepsilon),\beta,\varepsilon)>0$.
\end{restatable}

\begin{restatable}{mylem}{cexistbetafirst}
\label{lem:cexist_beta}
    Let the strictly leading $A$-eigenvalue $\mu_1$ be non-real with positive real part. Then for small enough values of $\varepsilon\geqslant0$ and $\alpha\in(0,1)$ there exists a differentiable function $\bcrit{\alpha,1}(\varepsilon)$ satisfying $\bcrit{\alpha,1}(0)=\tfrac{1-\alpha}{\mathrm{Re}(\mu_1)}$, such that, for $\beta=\bcrit{\alpha,1}(\varepsilon)$, $\nmax{1}$ and $\nmax{2}$ are given by Equation \eqref{eq:cmodul}. Furthermore
    \begin{equation}\label{eq:order_bound_beta}
      1-\alpha-\varepsilon<\mathrm{Re}(\mu_1)\bcrit{\alpha,1}(\varepsilon)<1-\alpha+\varepsilon,  
    \end{equation}
    and $\tfrac{\partial\fmax{1,2}}{\partial\beta}(\alpha,\bcrit{\alpha,1}(\varepsilon),\varepsilon)>0$.
\end{restatable}

We summarize these findings in the following definitions.

\begin{mydef}\label{def:acrit}
    For strictly leading $A$-eigenvalue $\mu_1$ with positive real part, and $\beta\in(0,\tfrac{1}{\mathrm{Re}(\mu_1)})$, the function $\acrit{\beta,1}$ defined in Lemmata \ref{lem:rexist} and \ref{lem:cexist} is called the $\alpha$-\emph{critical value function}.
\end{mydef}

\begin{mydef}\label{def:bcrit}
    For strictly leading $A$-eigenvalue $\mu_1$  with positive real part, and $\alpha\in(0,1)$, the function $\bcrit{\alpha,1}$ defined in Lemmata \ref{lem:rexist_beta} and \ref{lem:cexist_beta} is called the $\beta$-\emph{critical value function}.
\end{mydef}

\subsection{Dominant structure is preserved near critical values}

We now show that the fast dominant structure of matrix $A$ is inherited by model Jacobian $J_0$ when either model parameter is close enough to its corresponding critical value.

\begin{restatable}{mylem}{rdomfirst}
\label{lem:rdom}
Let the strictly leading $A$-eigenvalue $\mu_1$ be real, and $\acrit{\beta,1}$ be the $\alpha$-critical value function as in Definition \ref{def:acrit}. Then for small enough $\varepsilon>0$, $\beta\in(0,\tfrac{1}{\mathrm{Re}(\mu_1)})$, and $\alpha=\acrit{\beta,1}(\varepsilon)$, the associated $J_0$-eigenvalues $\nmin{j}$ and $\nmax{j}$ have negative real parts, for all $j\in\{2,\ldots,k\}$.
\end{restatable}

\begin{mylem}\label{lem:rdom_beta}
Let the strictly leading $A$-eigenvalue $\mu_1$ be real, and $\bcrit{\alpha,1}$ be the $\beta$-critical value function as in Definition \ref{def:bcrit}. Then for small enough $\varepsilon>0$, $\alpha\in(0,1)$, and $\beta=\bcrit{\alpha,1}(\varepsilon)$, the associated $J_0$-eigenvalues $\nmin{j}$ and $\nmax{j}$ have negative real parts, for all $j\in\{2,\ldots,k\}$.
\end{mylem}

\begin{restatable}{mylem}{cdomfirst}
\label{lem:cdom}
Let the strictly leading $A$-eigenvalue $\mu_1$ be non-real, and $\acrit{\beta,1}$ be the $\alpha$-critical value function as in Definition \ref{def:acrit}. Then for small enough $\varepsilon>0$, $\beta\in(0,\tfrac{1}{\mathrm{Re}(\mu_1)})$, and $\alpha=\acrit{\beta,1}(\varepsilon)$, the associated $J_0$-eigenvalues $\nmin{1}$, $\nmin{2}$, $\nmin{j}$ and $\nmax{j}$ have negative real parts, for all $j\in\{3,\ldots,k\}$.
\end{restatable}

\begin{mylem}\label{lem:cdom_beta}
Let the strictly leading $A$-eigenvalue $\mu_1$ be non-real, and $\bcrit{\alpha,1}$ be the $\beta$-critical value function as in Definition \ref{def:bcrit}. Then for small enough $\varepsilon>0$, $\alpha\in(0,1)$, and $\beta=\bcrit{\alpha,1}(\varepsilon)$, the associated $J_0$-eigenvalues $\nmin{1}$, $\nmin{2}$, $\nmin{j}$ and $\nmax{j}$ have negative real parts, for all $j\in\{3,\ldots,k\}$.
\end{mylem}

\section{Constructive rhythmic network control}\label{sec:synchro}

\subsection{Rhythmogenesis organized by a Hopf bifurcation}

From Lemmata~\ref{lem:rexist}, \ref{lem:cexist}, \ref{lem:rdom}, and~\ref{lem:cdom}, it follows that model~\eqref{eq:ei} undergoes a Hopf bifurcation through model parameter $\alpha$ at critical value $\alpha=\alpha_{\beta,1}(\varepsilon)$. We further characterize the dominant subspace $\mathcal H$ linearized dynamics at bifurcation, thus providing a description of the emerging rhythmic profile through the {\it Center Manifold Theorem}~\cite[Theorem 3.2.1]{guckenheimer83}. 

\begin{restatable}{mytheo}{rhythm}
\label{theo:rhythm}
    Consider model~\eqref{eq:ei}. Let Assumption~\ref{as:main leading} hold, $\beta\in(0,\tfrac{1}{\mathrm{Re}(\mu_1)})$, and $\acrit{\beta,1}$ be defined as Definition \ref{def:acrit}. Then, for $\varepsilon>0$ sufficiently small, the origin is exponentially stable for $\alpha<\alpha^*=\acrit{\beta,1}(\varepsilon)$, unstable for $\alpha>\alpha^*$, and a Hopf bifurcation occurs through $\alpha$ at $\alpha=\alpha^*$.
\end{restatable}

An analogous result for $\beta$ is achieved through Lemmata~\ref{lem:rexist_beta}, \ref{lem:cexist_beta}, \ref{lem:rdom_beta}, and~\ref{lem:cdom_beta}, at critical value $\beta=\bcrit{\alpha,1}(\varepsilon)$.

\begin{restatable}{mytheo}{rhythmbeta}
\label{theo:rhythm_beta}
    Consider model~\eqref{eq:ei}. Let Assumption~\ref{as:main leading} hold, $\alpha\in(0,1)$, and $\bcrit{\alpha,1}(\varepsilon)$ be defined as Definition \ref{def:bcrit}. Then, for $\varepsilon>0$ sufficiently small, the origin is exponentially stable for $\beta<\beta^*=\bcrit{\alpha,1}(\varepsilon)$, unstable for $\beta>\beta^*$, and a Hopf bifurcation occurs through $\beta$ at $\beta=\beta^*$.
\end{restatable}

The proofs for both results follow from the previous lemmata, and they are  presented in Appendix~\ref{app:theorems} of the online preprint~\cite{juarezalvarez2024analysis}. The negative real parts of the non-bifurcating eigenvalues and the 2-dominance of the linearized dynamics of model~\eqref{eq:ei} close to the Hopf bifurcation imply exponential stability of the 2-dimensional {\it center manifold}~\cite[Theorem 3.2.1]{guckenheimer83} of the bifurcation. Thus, the oscillatory behavior of the linearized dynamics inside the dominant subspace $\mathcal H$ characterizes the full nonlinear relative rhythmic profile emerging at the Hopf bifurcation (modulo errors of order $O\left((\alpha-\alpha_{\beta,1}(\varepsilon))^2\right)$ or $O\left((\beta-\beta_{\alpha,1}(\varepsilon))^2\right)$ for $\alpha$ or $\beta$ as bifurcation parameter). Observe that $\mathcal H$ is the generalized real eigenspace associated to the strictly leading complex eigenvalue pair. We state and prove the following proposition when $\alpha$ is the bifurcation parameter. The $\beta$ case is analogous.

\begin{myprop}\label{prop:synchro}
    Under the same assumptions as Theorem~\ref{theo:rhythm}, let $\alpha=\alpha_{\beta,1}(\varepsilon)$ and let $\vo{w}=(\vo{w}_x^t\vert\vo{w}_y^t)^t$, $A\vo{w}_x=\mu_1\vo{x}$, $\vo{w}_y=\tfrac{\varepsilon}{\varepsilon+\nu_1^+}\vo{w}_x$, be a right non-zero eigenvector of $J_0$ for the strictly leading purely complex $J_0$-eigenvalue $\nu_1^+$ associated to the strictly leading $A$-eigenvalue $\mu_1$. Write $\vo{w}_x=(\sigma_n e^{i\varphi_n})_{n=1}^{N}$ and suppose, without loss of generality, that $\sigma_1>0$ and $\sigma_1\geqslant \sigma_j$ for all $j\neq 1$. Then the solution $\vo{z}(t)=(\vo{x}^t(t)\vert\vo{y}^t(t))^t$ to the linear system $\dot{\vo{z}}=J_0\vo{z}$ with initial condition $\vo{z}(0)=c_1\mathrm{Re}(\vo{w})+c_2\mathrm{Im}(\vo{w})$ satisfies
    \begin{equation}\label{eq:synchro}
    \vo{x}_n(t)=\sigma_n(c_1\cos(\abs{\nmax{1}}t+\varphi_n)+c_2\sin(\abs{\nmax{1}}t+\varphi_n)),
    \end{equation}
    which corresponds to a relative rhythmic profile with relative amplitudes $\rho_j=\frac{\sigma_j}{\sigma_1}$ and relative phases $\theta_j=\varphi_j-\varphi_1$.
\end{myprop}

\begin{proof}
    Consider complex function $\boldsymbol{\zeta}(t)=e^{\nu_1^+t}\vo{w}$ as a solution to the IVP defined by $\dot{\vo{\zeta}}= J_0\vo{\zeta}$, $\vo{\zeta(0)}=\vo{w}$. By writing $\boldsymbol{\zeta}(t)=(\boldsymbol{\xi}^t(t)\vert\boldsymbol{\eta}^t(t))^t$, it is possible to find the analytic entry-wise solutions $\boldsymbol{\xi}_n(t)=\sigma_ne^{i(\abs{\nu_1^+}t+\theta_n)}$ for every $n\in\{1,\ldots,N\}$. Now write $\boldsymbol{\zeta}(t)=\mathrm{Re}(\boldsymbol{\zeta})(t)+i\mathrm{Im}(\boldsymbol{\zeta})(t)$. Then the solutions to real linear system $\dot{\vo{z}}=J_0\vo{z}$ are generated by $\mathrm{Re}(\boldsymbol{\zeta})(t)$ and $\mathrm{Im}(\boldsymbol{\zeta})(t)$. By hypothesis we have $\vo{z}(0)=c_1\mathrm{Re}(\vo{w})+c_2\mathrm{Im}(\vo{w})$, and therefore $\vo{z}(t)=c_1\mathrm{Re}(\boldsymbol{\zeta})(t)+c_2\mathrm{Im}(\boldsymbol{\zeta})(t)$ for every non-negative time. This in turn implies $\vo{x}(t)=c_1\mathrm{Re}(\boldsymbol{\xi})(t)+c_2\mathrm{Im}(\boldsymbol{\xi})(t)$, from whence~\eqref{eq:synchro} follows.
\end{proof}

\begin{restatable}{myrem}{osci_y}
    It is also easy to see, using $\vo{w}_y=\tfrac{\varepsilon}{\varepsilon+\nu_1^+}\vo{w}_x$, that the slow negative feedback variable $y_j$ oscillates with $\rho_{\varepsilon}$ times the amplitude of the oscillation of $x_j$, and with a phase difference relative to $x_j$ of $\theta_\varepsilon$, where $\rho_{\varepsilon}e^{i\theta_{\varepsilon}}=\tfrac{\varepsilon}{\varepsilon+\nu_1^+}$; in particular, $\rho_\varepsilon=O(\varepsilon)$.
\end{restatable}

By the Center Manifold Theorem~\cite[Theorem 3.2.1]{guckenheimer83}, the limit cycle emerging at the Hopf bifurcation lies on a two-dimensional manifold that is tangent to the dominant subspace $\mathcal H$ and is a small perturbation of one of the periodic solutions proved in Proposition~\ref{prop:synchro}. It follows that the leading eigenstructure of $A$ fully determines the relative rhythmic profile of the network rhythm emerging at the Hopf bifurcation.

\subsection{Stability of rhythmic profiles} \label{sec:crit}

In the following theorem we compute the coefficient $b$ in the normal form of the Hopf bifurcation (see Theorem~\ref{theo:hopf} in the online version~\cite{juarezalvarez2024analysis}), which determines the stability and the parametric region of existence of the limit cycle emerging at the bifurcation. The proof of this theorem is technical and is provided in the extended preprint.

\begin{restatable}{mytheo}{coeffgenfirst}
\label{theo:coeff_gen}
    Under the same assumptions as Theorem~\ref{theo:rhythm}, let $\alpha=\alpha_{\beta,1}(\varepsilon)$,  $\vo{v}=(\vo{v}_x^t\vert\vo{v}_y^t)^t$ be a left non-zero eigenvector of $J_0$ associated to purely complex $J_0$-eigenvalue $\nu_1^+$, and $\vo{w}=(\vo{w}_x^t\vert\vo{w}_y^t)^t$ be a right non-zero eigenvector of real matrix $J_0$ associated to purely complex eigenvalue $\cj{\nu}_1^+=-\nu_1^+$, such that $\vo{v}^t\vo{w}=0$ and $\cj{\vo{v}}^t\vo{w}=2$ (see Theorem~\ref{theo:datta}). Then coefficient $b$ in Theorem~\ref{theo:hopf} is given by
        \begin{align}\label{eq:hopf_coeff_b}
            b =&\tfrac{1}{16}\abs{1-\nmax{1}+\tfrac{\varepsilon}{\varepsilon-\nmax{1}}}^2S'''(0)\cdot\\
            &\cdot\mathrm{Re}\left(\left(1-\nmax{1}+\tfrac{\varepsilon}{\varepsilon-\nmax{1}}\right)\ip{\vo{v}_x}{\vo{w}_x\odot\vo{w}_x\odot\cj{\vo{w}}_x}\right).\nonumber
        \end{align}
\end{restatable}

\begin{figure}
    \centering
    \includegraphics[width=0.48\textwidth]{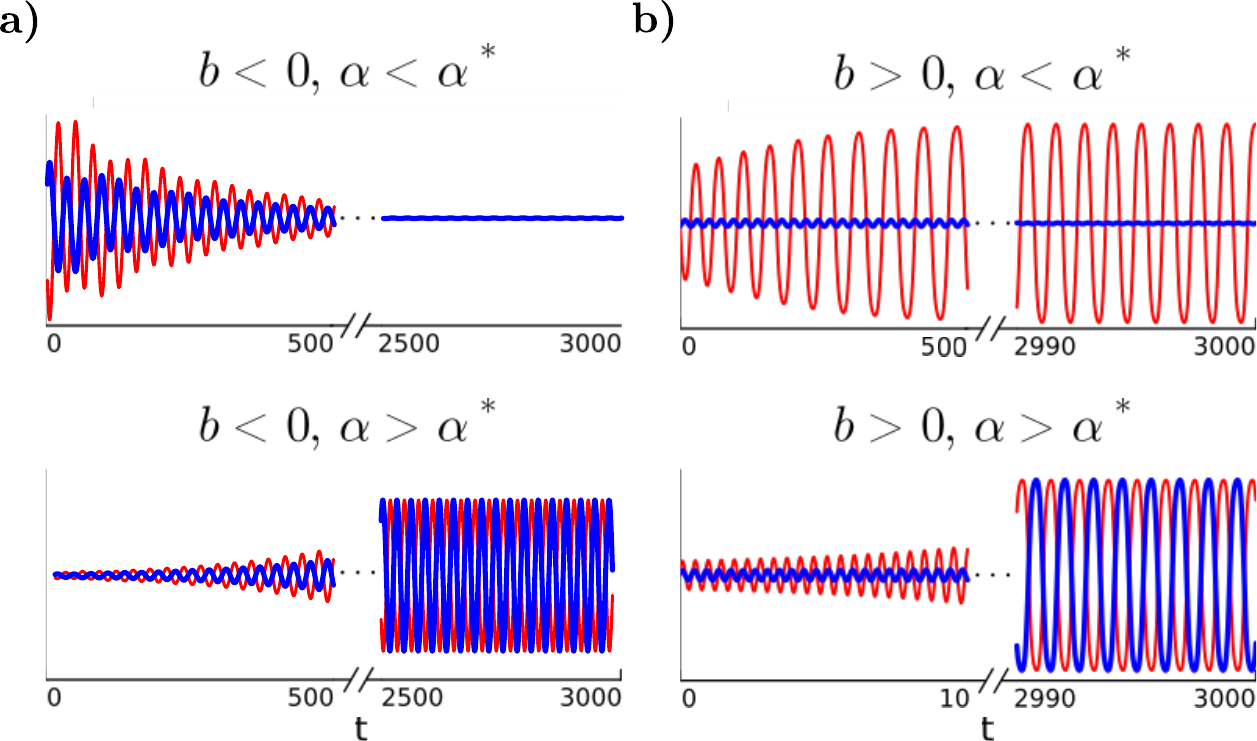}
    \caption{Dynamical behaviors close to a supercritical vs a subcritical Hopf bifurcation.
    \textbf{a)} Pre-bifurcation (top) and post-bifurcation (bottom) behaviors for the supercritical case. Pre-bifurcation, all trajectories exhibit damped oscillations converging to the origin. Post-bifurcation, all trajectories converge to a stable limit cycle oscillation emerging at bifurcation. \textbf{b)} Pre-bifurcation (top) and post-bifurcation (bottom) behaviors for the subcritical case. Pre-bifurcation, some trajectories exhibit damped oscillations converging to the origin while other trajectories converge toward a stable limit cycle. The two kinds of trajectories are separated by an unstable limit cycle emerging at bifurcation. Post-bifurcation, all non-equilibrium trajectories converge to the stable limit cycle.}
    \label{fig:crit_mix}
\end{figure}

When $b<0$ (supercritical Hopf bifurcation), the limit cycle emerging at the Hopf bifurcation is stable and exists for $\alpha$ or $\beta$ close to and above their critical values $\alpha^*$ or $\beta^*$, respectively. When $b>0$ (subcritical Hopf bifurcation), the limit cycle emerging at the Hopf bifurcation is unstable and exists for $\alpha$ or $\beta$ close to and below their critical values $\alpha^*$ or $\beta^*$, respectively. In the subcritical case, the unstable limit cycle is surrounded by a larger amplitude stable limit cycle that persists past the bifurcation point\footnote{The proof of this fact goes beyond the scope of this paper and involves invoking boundedness of the trajectories of model~\eqref{eq:ei} and computing higher-order derivatives of similar kinds as coefficients $a$ and $b$ in Theorem~\ref{theo:hopf}.}. Figure~\ref{fig:crit_mix} illustrates the qualitative difference between a supercritical and a subcritical Hopf bifurcation, as predicted by Theorem~\ref{theo:coeff_gen}.

We were not able to derive general conditions guaranteeing a given sign for $b$. However, the following corollary (proved in the online version \cite{juarezalvarez2024analysis}) provides a sufficient condition for the negativity of coefficient $b<0$, thus implying a supercritical Hopf bifurcation and stable limit cycle oscillations.

\begin{restatable}{mycor}{critfirst}
\label{cor:crit}
    If the strictly leading right eigenvectors of matrix $A$ are modulus-homogeneous, and $\varepsilon>0$ is small enough, then $b$ is negative, and the Hopf bifurcation undergone by system~\eqref{eq:ei} is supercritical.
\end{restatable}

One important special case of Corollary \ref{cor:crit} is when matrix $A$ is switching equivalent (see Section \ref{sec:math}) to either a positive in-regular or a non-negative in-regular irreducible matrix $A=MPM$. This case has been studied in previous works~\cite{juarez21}. As positive eigenvector $\vo{1}_N$ forms an eigenpair with positive eigenvalue $d>0$, then Perron-Frobenius theory~\cite[Theorem~8.4.4 and Exercise~8.4.P21]{horn} (see also Section~\ref{sec: fast dom}) implies that $d$ is the leading eigenvalue and that $A$ has a modulus-homogeneous leading eigenvector $M\vo{1}_N$. By applying Corollary~\ref{cor:crit}, we conclude that the rhythmic profile arising from the bifurcation must correspond to switching synchronization (see Definition~\ref{def:rp_types}), whose in-phase and anti-phase oscillators are determined by the signs of $M\vo{1}_N$.

\section{Designing rhythmic networks}\label{sec:build}

The results in the previous sections suggest a constructive way to design rhythmic networks with a desired rhythmic profile. Namely, if $\rho_2,\ldots,\rho_N$, all less than $\rho_1=1$, are the target relative amplitudes, and $\theta_2,\ldots,\theta_N$ are the target phase differences, it suffices to find an adjacency matrix $A$ such that $\vo{w}_x=(1,\rho_2e^{i\theta_2},\ldots,\rho_Ne^{i\theta_N})$ is a right leading eigenvector associated to a strictly leading eigenvalue $\mu_1$. Before discussing simple ways to achieve this, let us distinguish two important cases:
\begin{itemize}
    \item[i)] $\vo{w}_x\in\mathds{R}^N$, i.e. $\theta_j\in\{0,\pi\}\,\mathrm{mod}\,2\pi$ for all $j=2,\ldots,N$.
    \item[ii)] $\vo{w}_x\not\in\mathds{R}^N$, i.e. $\theta_j\not\in\{0,\pi\}\,\mathrm{mod}\,2\pi$ for at least one $j$.
\end{itemize}
In Case~{i)}, the strictly leading $A$-eigenvalue $\mu_1$ is real and the modulus of the associated strictly leading $J_0$-eigenvalues $\nu_1^{\pm}$, given by~\eqref{eq:rmodul}, is $O\left(\varepsilon^{1/2}\right)$. This implies that we cannot arbitrarily control the period $T=2\pi\abs{\nu_1^{\pm}}^{-1}$ of the emerging rhythmic profile, which diverges to infinity as $\varepsilon\to 0$. However, in practice, given a sufficiently large timescale separation, that is, a sufficiently small {\it fixed} $\varepsilon$, we can achieve a desired period by suitably scaling the model vector field, i.e., by suitably speeding up all the model variables. This problem is absent in Case~{ii)} because $\mu_1$ is not real and, using~\eqref{eq:rmodul}, the modulus of the strictly leading $J_0$-eigenvalues $\nu_{1,2}^{+}$ is approximately $\beta{\rm Im}(\mu_1)>0 + O\left(\varepsilon^{1/2}\right)$. Thus, when the leading eigenstructure of $A$ is not real, and therefore relative phases of the emerging rhythm are not constrained to be $0$ or $\pi$, then the emerging rhythm period is approximately $T\approx 2\pi(\beta{\rm Im}(\mu_1))^{-1}$, which is fully controllable by suitably designing $\mu_1$.

We illustrate the construction of matrix $A$ on two specific rhythmic profile control problems: {\it amplitude control}, as an example of Case i), and {\it phase control}, as an example of Case ii). We consider the case in which the network topology is unconstrained and discuss extensions to the constrained case in Section~\ref{sec:disc}.

\subsection{Rhythm amplitude control}

Our aim is to achieve a network rhythm in which oscillations are either in-phase or anti-phase, i.e., $\theta_2,\cdots,\theta_N\in\{0,\pi\}\,\mathrm{mod}\,2\pi$, but with different desired relative oscillation amplitudes $\rho_2,\ldots,\rho_N$. With an abuse of terminology, we allow the amplitude $\rho_j$ to be negative, which is equivalent to setting $\theta_j=\pi$ (anti-phase oscillations), but still impose the constraint $|\rho_j|\leqslant \rho_1=1$. A simple way to build an adjacency matrix $A$ leading to such a relative rhythmic profile is the following:
\begin{enumerate}
    \item Let $\vo{w}_x=(1,\rho_2,\ldots,\rho_N)$.
    \item Pick $\mu_1>0$, $\mu_2,\ldots,\mu_N<\mu_1$, and let $D={\rm diag}(\mu_1,\mu_2,\ldots,\mu_N)$.
    \item Find an ordered basis ${\mathcal B}=\{\vo{w}_x,\vo{u}_2,\ldots,\vo{u}_N\}$ of $\mathds{R}^N$ and let $Q=(\vo{w}_x\ \vo{u}_2\ \cdots\ \vo{u}_N)$ be the change of variable from the canonical basis of $\mathds{R}^N$ to $\mathcal B$.
    \item Define $A=QDQ^{-1}$.
\end{enumerate}

The constructed $A$ has $\mu_1$ as its leading eigenvalue, and $\vo{w}_x$ as a leading eigenvector. Since $(\vo{w}_x)_1=1$, a possibility to build the basis $\mathcal B$ is to pick $\vo{u}_j=\vo{e}_j$, which leads to
\[
A=\left(\begin{array}{cccc}
    \mu_1 & 0 & \ldots & 0  \\
     w_2(\mu_1-\mu_2) & \mu_2 & \ldots & 0\\
     \vdots & \vdots & \ddots & \vdots \\
     w_N(\mu_1-\mu_N) & 0 & \ldots & \mu_N
\end{array}\right),
\]
where $w_j:=(\vo{w}_x)_j$. Such an adjacency matrix defines a star topology, where oscillations are driven by the center node $(x_1,y_1)$. Observe that each oscillator has also a self-loop with weight $\mu_j$. The result of this design is illustrated in Figure~\ref{fig:design}, top.

\begin{figure}
    \centering    
    \includegraphics[width=0.48\textwidth]{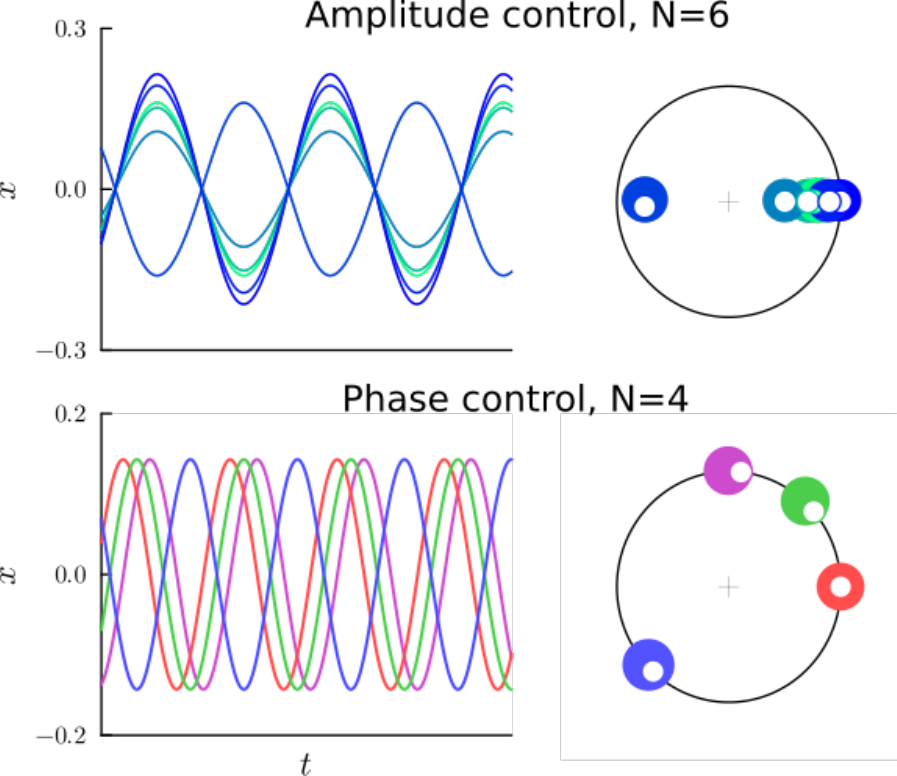}
    \caption{Examples of controlled rhythmic profiles with matrices built through the algorithm described in this section. \emph{Top:} 
    a real leading eigenvector, with some entries being possibly of opposite signs, results in in-phase or anti-phase oscillations. \emph{Below:} a modulus-homogeneous leading eigenvector where at least one entry is not real leads to a shifting synchronization behavior.}
    \label{fig:design}
\end{figure}

\subsection{Rhythm phase control}

Our aim is to achieve a network rhythm in which oscillators have the same amplitude, i.e., $\rho_2=\cdots=\rho_N=1$, but with non-zero and non-$\pi$ desired relative phases $\theta_2,\ldots,\theta_N$. A simple way to build an adjacency matrix $A$ satisfying these constraints is the following:
\begin{enumerate}
    \item Let $\vo{w}_x=(1,e^{i\theta_2},\ldots,e^{i\theta_N})$. Observe that, if $\theta_j\not\in\{0,\pi\}\,\mathrm{mod}\,2\pi$ for any $j>1$, then $\vo{w}_x$ and $\overline{\vo{w}}_x$ are linearly independent.
    \item Pick $\mu_1=u_1+iv_1$ with $u_1,v_1>0$, $\mu_2=\cj{\mu}_1=u_1-iv_1$, and $\mu_3,\ldots,\mu_N\in\mathds{R}$, $\mu_j<u_1$. Let
$D= {\rm diag}(\mu_1,\overline{\mu}_1,\mu_3,\ldots,\mu_N)$.
\item Find an ordered basis ${\mathcal B}=\{\vo{w}_x,\overline{\vo{w}}_x,\vo{u}_3,\ldots,\vo{u}_N\}$ of $\mathds{C}^N$ such that $\vo{u}_3,\ldots,\vo{u}_N$ are real, and let $Q=(\vo{w}_x\ \cj{\vo{w}}_x\ \vo{u}_3\ \cdots\ \vo{u}_N)$ be the change of variable from the canonical basis of $\mathds{C}^N$ to $\mathcal B$.
\item Define $A=QDQ^{-1}$. Note that $A$ is real because $D$ is the complex Jordan form associated to the complex basis ${\mathcal B}$, and $Q$ is the change of basis that puts $A$ in its complex Jordan form.
\end{enumerate}

Observe that $A$ has $\mu_1$, $\mu_2=\cj{\mu}_1$ as strictly leading eigenvalues with $\vo{w}_x$ and $\cj{\vo{w}}_x$ as strictly leading eigenvectors. If $\theta_2\not\in\{0,\pi\}\,\mathrm{mod}\,2\pi$ then $\vo{w}_2\in\mathds{C}\backslash\mathds{R}$, and thus $\mathrm{Im}(\vo{w}_2)\not=0$. Therefore, a possibility to build the basis $\mathcal B$ is to pick $\vo{u}_j=\vo{e}_j$, which leads to adjacency matrix A given by
\[
\left(\begin{array}{ccccc}
    \tfrac{\mathrm{Im}(\cj{\mu}_1w_2)}{\mathrm{Im}(w_2)} & \tfrac{\mathrm{Im}(\mu_1)}{\mathrm{Im}(w_2)} & 0 & \ldots & 0  \\
    \abs{w_2}^2\tfrac{\mathrm{Im}(\cj{\mu}_1)}{\mathrm{Im}(w_2)} & \tfrac{\mathrm{Im}(\mu_1w_2)}{\mathrm{Im}(w_2)} & 0 & \ldots & 0  \\
    \tfrac{\mathrm{Im}((\cj{\mu}_1-\mu_3)w_2\cj{w}_3)}{\mathrm{Im}(w_2)} & \tfrac{\mathrm{Im}((\mu_1-\mu_3)w_3)}{\mathrm{Im}(w_2)} & \mu_3 & \ldots & 0  \\
    \vdots & \vdots & \vdots & \ddots & \vdots \\
    \tfrac{\mathrm{Im}((\cj{\mu}_1-\mu_N)w_2\cj{w}_N)}{\mathrm{Im}(w_2)} & \tfrac{\mathrm{Im}((\mu_1-\mu_N)w_N)}{\mathrm{Im}(w_2)} & 0 & \ldots & \mu_N
\end{array}\right)
\]
where $w_j:=(\vo{w}_x)_j$. Such an adjacency matrix corresponds to a star topology with a two-node core, in which the first two oscillators are mutually coupled and drive all the others. Observe that each oscillator may also exhibit a self-loop. The result of this design is illustrated in Figure~\ref{fig:design}, bottom.

\section{Discussion}\label{sec:disc}

We introduced new theoretical tools to design the rhythmic profile of a rhythmic network.  Our tools are constructive and can be used for analysis, control, and design. Furthermore, they are developed on a model that is compatible with neuromorphic engineering applications.

Future theoretical development for the extension of the main results includes passing from local to global analysis and considering more complicated (e.g., higher-dimensional or with nonlinear slow negative feedback) node dynamics. Another important theoretical development, related to the design strategy described in Section~\ref{sec:build}, is to consider the case in which the network is structured, i.e., only some edges are present and only the weights of those edges can be tuned to impose a desired leading eigenstructure. 

It is a question that we did not address here due to space constraints, and for the sake of conciseness, whether the mixed-feedback structure is necessary to ensure a simple and general mapping from the network adjacency matrix to the resulting rhythmic pattern. A natural model in which this question could be addressed is the Stuart-Landau oscillator~\cite{provansal87,aoyagi95}, also known as the Andronov-Hopf oscillator~\cite{izhikevich}, that is, the truncated normal form~\cite[Equation (3.4.8)]{guckenheimer83} of the Hopf bifurcation~\cite{andronov}. This model is given by a simple, third order polynomial expression which makes it widely used for the analysis of rhythmic networks \cite{aronson90,hoppensteadt96,tukhlina07,panteley15,panteley15SL,premalatha18,harrison23}. At the same time, it is not evidently organized by a mixed-feedback structure, making it the perfect candidate to understand to which extent the mixed-feedback structure simplifies or in which sense it is necessary for controlling the rhythmic profile of an oscillator network.

Applications include the hardware realization of model~\eqref{eq:ei} in neuromorphic electronics, the use of the resulting tunable rhythmic controller for locomotion in simple legged robots, as well as its application to biological rhythmogenesis in complex neuron networks such as the suprachiasmatic nucleus~\cite{juarez23}.

\bibliographystyle{elsarticle-num} 
\bibliography{collective}

\include{appendix1}

\end{document}

%% file: appendix1.tex
\clearpage
\onecolumn
\appendices
\pagenumbering{alph}

\section{Some useful results}\label{app:results}

The following theorem (which will be used for bifurcation theoretical computations) follows from methods used for classic biorthogonality results on the eigenvectors of matrices $A$ and $A^t$~\cite[Theorem 7.7]{datta}.

\begin{mytheo}\label{theo:datta}
     Let $\lambda$ be a simple non-real eigenvalue with right eigenvector $\vo{w}$. Then there exists a left eigenvector $\vo{v}$ associated to $\cj{\lambda}$ such that $\vo{v}^t\vo{w}=0$ and $\cj{\vo{v}}^t\vo{w}\not=0$.
\end{mytheo}

Our main results will largely rely on Hopf bifurcation theory. A Hopf bifurcation describes the emergence of limit cycles in a parameterized vector field as a (control) parameter crosses a critical value.
The following theorem (from \cite[Theorem 3.4.2]{guckenheimer83},\cite[Chapter VIII, Proposition 3.3]{golubitsky}) formalizes this idea.

\begin{mytheo}\label{theo:hopf}
    Suppose that model $\dot{\vo{x}}=\vo{f}(\vo{x},\alpha)$, $\vo{x}\in\mathds{R}^N$, $\alpha\in\mathds{R}$, has an equilibrium point at $(\vo{x}_0,\alpha_0)$.
    If the model Jacobian $J$ at $(\vo{x}_0,\alpha_0)$ has a simple pair of pure imaginary eigenvalues and all other eigenvalues with non-zero real parts
    then there exists a smooth curve of equilibria $(\vo{x}(\alpha),\alpha)$ such that $\vo{x}(\alpha_0)=\vo{x}_0$, and eigenvalues $\lambda(\alpha)$, $\cj{\lambda}(\alpha)$ which are pure imaginary at $\alpha=\alpha_0$ vary smoothly with $\alpha$. Let $\vo{v}$ and $\vo{w}$ be the left and right eigenvectors of $J$ at $(\vo{x}_0,\alpha_0)$, satisfying $\vo{v}^tJ=i\abs{\lambda}\vo{v}^t$ and $J\vo{w}=-i\abs{\lambda}\vo{w}$, $\lambda\not=0$, such that $\cj{\vo{v}}^t\vo{w}=2$ and $\vo{v}^t\vo{w}=0$.
    Let $a=\frac{\partial\mathrm{Re}(\lambda)}{\partial\alpha}(\vo{x_0},\alpha_0)$
    and $b=\tfrac{1}{16}\mathrm{Re}\left(\ip{\vo{v}}{(\mathrm{d}^3\vo{f})_{\vo{x}_0,\alpha_0}(\vo{w},\vo{w},\cj{\vo{w}})}\right)$.
    If $a>0$ then for $\alpha$ sufficiently close to $\alpha_0$, the equilibrium $\vo{x}(\alpha)$ is stable for $\alpha<\alpha_0$ and unstable for $\alpha>\alpha_0$. Furthermore, if $b<0$, then for $\alpha>\alpha_0$, there exists a stable limit cycle solution $\vo{l}_{\alpha}(t)$ satisfying $\max_{t\in\mathds R}\|\vo{x}_0-\vo{l}_{\alpha}(t)\|=O\left((\alpha-\alpha_0)^{1/2}\right)$. Conversely, if $b>0$, then for $\alpha<\alpha_0$ there exists an unstable limit cycle solution $\vo{l}_{\alpha}(t)$ satisfying $\max_{t\in\mathds R}\|\vo{x}_0-\vo{l}_{\alpha}(t)\|=O\left((\alpha_0-\alpha)^{1/2}\right)$. (The results for $a<0$ are omitted for conciseness as they won't be used.)
\end{mytheo}

In the case of our model, we will further prove that all eigenvalues, other than the bifurcating pair, have negative real parts, which guarantees convergence to a certain invariant manifold, tightly related to the behavior to seek to achieve.

Figure~\ref{fig:bifcontrast} illustrates this theorem for the case $a>0$, $b<0$.

\section{Proofs of theorems}\label{app:theorems}

\rhythm*

\begin{proof}
    By Lemmata \ref{lem:rexist} and \ref{lem:cexist}, $J_0$ has a pair of pure imaginary non-real complex eigenvalues $\nu_1^\pm$ (if $\mu_1$ is real) or $\nmax{1,2}$ (if $\mu_1$ is non-real) satisfying $\tfrac{\partial\mathrm{Re}(\nmax{1})}{\partial\alpha}(\acrit{\beta,1}(\varepsilon),\beta,\varepsilon)>0$ at $\alpha=\acrit{\beta,1}(\varepsilon)$. By Lemmata \ref{lem:rdom} and \ref{lem:cdom}, all other $J_0$-eigenvalues have negative real part. Thus, all the conditions of the Hopf Bifurcation Theorem~\ref{theo:hopf} are satisfied for the case $a>0$.
\end{proof}

\rhythmbeta*

\begin{proof}
    By Lemmata \ref{lem:rexist_beta} and \ref{lem:cexist_beta}, $J_0$ has a pair of pure imaginary non-real complex eigenvalues $\nu_1^\pm$ (if $\mu_1$ is real) or $\nmax{1,2}$ (if $\mu_1$ is non-real) satisfying $\tfrac{\partial\mathrm{Re}(\nmax{1})}{\partial\beta}(\alpha,\bcrit{\alpha,1}(\varepsilon),\varepsilon)>0$ at $\beta=\bcrit{\alpha,1}$. By Lemmata \ref{lem:rdom_beta} and \ref{lem:cdom_beta}, all other $J_0$-eigenvalues have negative real part. Thus, all the conditions of the Hopf Bifurcation Theorem~\ref{theo:hopf} are satisfied for the case $a>0$.
\end{proof}

\section{Proofs of lemmata}\label{app:proofs}

\rquadfirst*

\begin{proof}
    Start by observing that $\lambda^*=-\varepsilon$ does not satisfy quadratic condition \eqref{eq:rquad} for $\varepsilon\not=0$. We first prove that conditions \eqref{eq:rquad} and \eqref{eq:rquad_eq} are equivalent. Indeed,
    \begin{align*}
    \beta\mu=1-\alpha+\lambda+\tfrac{\varepsilon}{\varepsilon+\lambda}\Leftrightarrow\beta\mu(\varepsilon+\lambda)=(\lambda+1-\alpha)(\lambda+\varepsilon)+\varepsilon\Leftrightarrow\beta\mu\lambda+\varepsilon\beta\mu=\lambda^2+(1+\varepsilon-\alpha)\lambda+\varepsilon(2-\alpha).    
    \end{align*}
    
    Given $\mu\in\sigma(A)$ and a non-zero right eigenvector $\vo{w}_x\in\mathds{C}^N$, consider $\lambda\in\mathds{C}$ any complex number satisfying \eqref{eq:rquad}, or equivalently \eqref{eq:rquad_eq}.
    By proposing $\vo{w}=(\vo{w}_x^t\vert\tfrac{\varepsilon}{\varepsilon+\lambda}\vo{w}_x^t)^t\in\mathds{C}^{2N}$ it suffices to show that $J_0\vo{w}=\lambda\vo{w}$. Certainly,
    \begin{align*}
        J_0\vo{w}&=\left(\begin{array}{c|c}
        (\alpha-1)I_N+\beta A & -I_N  \\
        \hline
        \varepsilon I_N & -\varepsilon I_N 
    \end{array}\right)\left(\begin{array}{c}
        \vo{w}_x   \\
        \hline
        \tfrac{\varepsilon}{\varepsilon+\lambda}\vo{w}_x
    \end{array}\right)=\left(\begin{array}{c}
        (\alpha-1)\vo{w}_x+\beta A\vo{w_x}-\tfrac{\varepsilon}{\varepsilon+\lambda}\vo{w}_x   \\
           \hline
        \varepsilon\vo{w}_x-\tfrac{\varepsilon^2}{\varepsilon+\lambda}\vo{w}_x  
    \end{array}\right)\\
    &=\left(\begin{array}{c}
        (\alpha-1-\tfrac{\varepsilon}{\varepsilon+\lambda})\vo{w}_x+\beta\mu\vo{w}_x   \\
           \hline
        (\varepsilon-\tfrac{\varepsilon^2}{\varepsilon+\lambda})\vo{w}_x  
    \end{array}\right)
    =\left(\begin{array}{c}
        (\lambda-\beta\mu)\vo{w}_x+\beta\mu\vo{w}_x   \\
           \hline
        \tfrac{\varepsilon^2+\varepsilon\lambda-\varepsilon^2}{\varepsilon+\lambda}\vo{w}_x  
    \end{array}\right)=\left(\begin{array}{c}
        \lambda\vo{w}_x   \\
           \hline
        \lambda\tfrac{\varepsilon}{\varepsilon+\lambda}\vo{w}_x  
    \end{array}\right)=\lambda\vo{w}.    
    \end{align*}
    Conversely, suppose that $\lambda\in\sigma(J_0)$ is an eigenvalue with an associated non-zero eigenvector $\vo{w}=(\vo{w}_x^t\vert\vo{w}_y^t)^t\in\mathds{C}^{2N}$ so that $(J_0-\lambda I_{2N})\vo{w}=\vo{0}_{2N}$. This translates to
    \begin{align*}
        \vo{0}_{2N}=\left(\begin{array}{c|c}
        (\alpha-1-\lambda)I_N+\beta A & -I_N  \\
        \hline
        \varepsilon I_N & -(\varepsilon+\lambda)I_N 
    \end{array}\right)\left(\begin{array}{c}
        \vo{w}_x   \\
        \hline
        \vo{w}_y  
    \end{array}\right)=\left(\begin{array}{c}
        (\alpha-1-\lambda)\vo{w}_x+\beta A\vo{w}_x-\vo{w}_y   \\
           \hline
        \varepsilon\vo{w}_x-(\varepsilon+\lambda)\vo{w}_y 
    \end{array}\right),
    \end{align*}

from whence it is seen that $\lambda^*=-\varepsilon$ is not a right eigenvalue as $\varepsilon\not=0$. The last equality is equivalent to the system of vector equations
\begin{align*}
    \beta A\vo{w}_x&=(1-\alpha+\lambda+\tfrac{\varepsilon}{\varepsilon+\lambda_1})\vo{w}_x,\\
        \vo{w}_y&=\tfrac{\varepsilon}{\varepsilon+\lambda}\vo{w}_x,
\end{align*}
so that $\mu=\tfrac{1}{\beta}(1-\alpha+\lambda+\tfrac{\varepsilon}{\varepsilon+\lambda})$ must be an eigenvalue for matrix $A$. Since this has already been seen to be equivalent to quadratic condition \eqref{eq:rquad}, it concludes the proof.
\end{proof}

\lquadfirst*

\begin{proof}
    Given $\mu\in\sigma(A)$ and a non-zero vector $\vo{v}_x\in\mathds{C}^N$ such that $\vo{v}_x^tA=\mu\vo{v}_x^t$, consider $\lambda\in\mathds{C}\backslash\{-\varepsilon\}$ any complex number satisfying \eqref{eq:rquad_eq}. By proposing $\vo{v}=(\vo{v}_x^t\vert\tfrac{-1}{\varepsilon+\lambda}\vo{v}_x^t)^t\in\mathds{C}^{2N}$, it suffices to show that $\vo{v}^tJ_0=\lambda\vo{v}^t$. Certainly,
    \begin{align*}
        \vo{v}^tJ_0&=(\vo{v}_x^t\vert\tfrac{-1}{\varepsilon+\lambda}\vo{v}_x^t)\left(\begin{array}{c|c}
        (\alpha-1)I_N+\beta A & -I_N  \\
        \hline
        \varepsilon I_N & -\varepsilon I_N 
    \end{array}\right)=\left(\begin{array}{c}
        (\alpha-1)\vo{v}_x^t+\beta A\vo{v}_x^t-\tfrac{\varepsilon}{\varepsilon+\lambda}\vo{v}_x^t   \\
           \hline
        -\vo{v}_x^t+\tfrac{\varepsilon}{\varepsilon+\lambda}\vo{v}_x  
    \end{array}\right)^t\\
    &=\left(\begin{array}{c}
        (\alpha-1-\tfrac{\varepsilon}{\varepsilon+\lambda})\vo{v}_x^t+\beta\mu\vo{v}_x^t   \\
           \hline
        (-1+\tfrac{\varepsilon}{\varepsilon+\lambda})\vo{v}_x^t  
    \end{array}\right)^t=\left(\begin{array}{c}
        (\lambda-\beta\mu)\vo{v}_x^t+\beta\mu\vo{v}_x^t   \\
           \hline
        \tfrac{-\lambda}{\varepsilon+\lambda}\vo{v}_x^t 
    \end{array}\right)^t=\lambda(\vo{v}_x^t\vert\tfrac{-1}{\varepsilon+\lambda}\vo{v}_x^t)=\lambda\vo{v}^t.
    \end{align*}
    Conversely, suppose that $\lambda\in\sigma(J_0)$ is a left eigenvalue with an associated non-zero eigenvector $\vo{v}=(\vo{v}_x^t\vert\vo{v}_y^t)^t\in\mathds{C}^{2N}$ so that $\vo{v}^t(J_0-\lambda I_{2N})=\vo{0}^t_{2N}$. This translates to
    \begin{align*}
    \vo{0}_{2N}^t&=(\vo{v}_x^t\vert\vo{v}_y^t)\left(\begin{array}{c|c}
        (\alpha-1-\lambda)I_N+\beta A & -I_N  \\
        \hline
        \varepsilon I_N & -(\varepsilon+\lambda)I_N 
    \end{array}\right)=\left(\begin{array}{c}
        (\alpha-1-\lambda)\vo{v}_x^t+\beta A\vo{v}_x+\varepsilon\vo{v}_y^t   \\
           \hline
        -\vo{v}_x^t-(\varepsilon+\lambda)\vo{v}_y^t  
    \end{array}\right)^t,    
    \end{align*}
    
from whence it is seen that $\lambda^*=-\varepsilon$ is not an eigenvalue. The last equality is equivalent to the vector equation system
\begin{align*}
    \beta A\vo{v}_x&=(1-\alpha+\lambda+\tfrac{\varepsilon}{\varepsilon+\lambda})\vo{v}_x,\\
    \vo{v}_y&=-\tfrac{1}{\varepsilon+\lambda}\vo{v}_x,
\end{align*}
so that $\mu=\tfrac{1}{\beta}(1-\alpha+\lambda+\tfrac{\varepsilon}{\varepsilon+\lambda})$ must be an eigenvalue for matrix $A$. This concludes the proof.
\end{proof}

\conjfirst*

\begin{proof}

    In the first case we assume that $\mu_j$, as well as the two distinct solutions of
    $$  \lambda^2+(1+\varepsilon-\alpha-\beta\mu)\lambda+\varepsilon(2-\alpha-\beta\mu)=0,  $$
    are real. This gives straightforward expressions for
    \begin{align*}
        \nu^{+}_j&=\dfrac{\alpha+\beta\mu-1-\varepsilon+\sqrt{(\alpha+\beta\mu-1-\varepsilon)^2-4\varepsilon(2-\alpha-\beta\mu)}}{2},\\
        \nu^{-}_j&=\dfrac{\alpha+\beta\mu-1-\varepsilon-\sqrt{(\alpha+\beta\mu-1-\varepsilon)^2-4\varepsilon(2-\alpha-\beta\mu)}}{2}.
    \end{align*}
    Now observe that $\nu^-_j<\nu^+_j$ if and only if the discriminant $(\alpha+\beta\mu-1-\varepsilon)^2-4\varepsilon(2-\alpha-\beta\mu)$ is non-zero. If $\alpha+\beta\mu=1$, then the discriminant is reduced to $\varepsilon^2-4\varepsilon$, which is non-zero for $\varepsilon\in(0,4)$. If $\alpha+\beta\mu\not=1$, then the discriminant is strictly positive for $\varepsilon=0$, and therefore it is kept strictly positive for small enough values of $\varepsilon$, thus completing this part of the proof. In the second case it suffices to see that, for $\mu_j\in\mathds{R}$, eigenvalues $\nu_{\mu}^{\pm}=\Phi_{\mu}\pm i\Psi_{\mu}$ as given by \eqref{eq:lead2} are complex numbers with real parts equal to $\tfrac{1}{2}(\alpha+\beta\mu-1-\varepsilon)$, and with imaginary parts of opposite signs. Finally, let $\mu_j=u_j+iv_j$ with $v_j\not=0$, and by hypothesis $\mu_{j+1}=u_j-iv_j$. By conjugating condition \eqref{eq:rquad}, one gets
    \begin{align*}
        0&=\cj{\lambda}^2+(1+\varepsilon-\alpha-\beta\cj{\mu}_j)\cj{\lambda}+\varepsilon(2-\beta\cj{\mu}_j-\alpha=\cj{\lambda}^2+(1+\varepsilon-\alpha-\beta\mu_{j+1})\cj{\lambda}+\varepsilon(2-\alpha-\beta\mu_{j+1}),
    \end{align*}
    which concludes that  $\cj{\lambda}_{j}$ satisfies condition \eqref{eq:rquad} written for $\mu_{j+1}$, where $\lmax{j}$ is a solution for condition \eqref{eq:rquad} written for $\mu_{j}$. It is straightforward to verify from Equations \eqref{eq:lead2} that $\fmax{j}=\fmax{j+1}$ as $v_j$ is always squared in them. The sign term $d$ for $\pmax{j+1}$ is given by
    \begin{align*}
      &\mathrm{sgn}(\beta(v_{j+1})(\alpha+\beta u_{j+1}-1-\varepsilon))=\mathrm{sgn}(\beta(-v_j)(\alpha+\beta u_j-1-\varepsilon))=-\mathrm{sgn}(\beta v_j(\alpha+\beta u_j-1-\varepsilon)).  
    \end{align*}
    Then it is clear that $\pmax{j+1}=-\pmax{j}$, which implies that $\nmax{j}$ and $\nmax{j+1}$ are conjugates. Equation $\cjnmin{j}=\nmin{j}$ is similarly verified, so we conclude the proof. 
\end{proof}

\rexistfirst*

\begin{proof}
    When evaluating at $\beta>0$, $\varepsilon\in(0,1)$, $\alpha=\acrit{\beta,1}(\varepsilon)$, eigenvalues $\nu_{\mu}^{\pm}=$ as given by \eqref{eq:lead2} are reduced to $\pm i\sqrt{\varepsilon(1-\varepsilon)}$, so that $\nmin{1}$ and $\nmax{1}$ are conjugate non-real numbers whenever $\varepsilon\in(0,1)$. Discriminant $\Delta=(\alpha+\beta\mu_1-1-\varepsilon)^2-4\varepsilon(2-\beta\mu_1-\alpha)$ reduces to $-4\varepsilon(1-\varepsilon)<0$ at $\beta>0$, $\varepsilon>0$, $\alpha=\acrit{\beta,1}(\varepsilon)$, so by continuity there must exist an open non-empty subset $V$ of the parameter space where the real part of $\nmax{1}$ is given by $\fmax{1}=\tfrac{1}{2}(\alpha+\beta\mu_1-1-\varepsilon)$, and thus $\tfrac{\partial\fmax{1}}{\partial\alpha}(\acrit{\beta,1}(\varepsilon),\beta,\varepsilon)=\tfrac{1}{2}>0$.
\end{proof}

\rexistbetafirst*

\begin{proof}
    Take $\bcrit{\alpha,1}(\varepsilon)=\tfrac{1}{\mu_1}(1+\varepsilon-\alpha)$. Proceed as in Lemma \ref{lem:rexist}, and check that $\tfrac{\partial\fmax{1}}{\partial\beta}(\alpha,\bcrit{\alpha,1}(\varepsilon),\varepsilon)=\tfrac{\mu_1}{2}\not=0$.
\end{proof}

\begin{figure*}
\begin{equation}\label{eq:lead2}    
    \begin{split}
        \fmax{\mu}(\alpha,\beta,\varepsilon)&=\mbox{\tiny{$\dfrac{c}{2}+\dfrac{1}{2}\sqrt{\dfrac{\sqrt{((\alpha+\beta u-1-\varepsilon)^2-\beta^2v^2-4\varepsilon(2-\alpha-\beta u))^2+4\beta^2v^2(\alpha+\beta u-1+\varepsilon)^2}+(\alpha+\beta u-1-\varepsilon)^2-\beta^2v^2-4\varepsilon(2-\alpha-\beta u)}{2}}$}},\\
        \pmax{\mu}(\alpha,\beta,\varepsilon)&=\mbox{\tiny{$\dfrac{\beta v}{2}+\dfrac{d}{2}\sqrt{\dfrac{\sqrt{((\alpha+\beta u-1-\varepsilon)^2-\beta^2v^2-4\varepsilon(2-\alpha-\beta u))^2+4\beta^2v^2(\alpha+\beta u-1+\varepsilon)^2}-(\alpha+\beta u-1-\varepsilon)^2+\beta^2v^2+4\varepsilon(2-\alpha-\beta u)}{2}}$}},\\
        \fmin{\mu}(\alpha,\beta,\varepsilon)&=\mbox{\tiny{$\dfrac{c}{2}-\dfrac{1}{2}\sqrt{\dfrac{\sqrt{((\alpha+\beta u-1-\varepsilon)^2-\beta^2v^2-4\varepsilon(2-\alpha-\beta u))^2+4\beta^2v^2(\alpha+\beta u-1+\varepsilon)^2}+(\alpha+\beta u-1-\varepsilon)^2-\beta^2v^2-4\varepsilon(2-\alpha-\beta u)}{2}}$}},\\
        \pmin{\mu}(\alpha,\beta,\varepsilon)&=\mbox{\tiny{$\dfrac{\beta v}{2}-\dfrac{d}{2}\sqrt{\dfrac{\sqrt{((\alpha+\beta u-1-\varepsilon)^2-\beta^2v^2-4\varepsilon(2-\alpha-\beta u))^2+4\beta^2v^2(\alpha+\beta u-1+\varepsilon)^2}-(\alpha+\beta u-1-\varepsilon)^2+\beta^2v^2+4\varepsilon(2-\alpha-\beta u)}{2}}$}}.
    \end{split}
\end{equation} 
\end{figure*}

\cexistfirst*

\begin{proof}
    Take $\mu_1=u+iv$. By formulae \eqref{eq:lead2}, equation $\fmax{\mu}=0$ is equivalent to setting $\alpha+\beta u-1-\varepsilon$ equal to

\begin{equation}\label{eq:fmax_zero}
    -\sqrt{\dfrac{\mbox{\tiny{$\sqrt{((\alpha+\beta u-1-\varepsilon)^2-\beta^2v^2-4\varepsilon(2-\alpha-\beta u))^2+4\beta^2v^2(\alpha+\beta u-1+\varepsilon)^2}+(\alpha+\beta u-1-\varepsilon)^2-\beta^2v^2-4\varepsilon(2-\alpha-\beta u)$}}}{2}},
\end{equation}
    which in turn implies
\begin{align}
    \label{eq:implicit_alpha_og}
    &
    \scriptstyle{(\alpha+\beta u-1-\varepsilon)^2=\tfrac{\sqrt{((\alpha+\beta u-1-\varepsilon)^2-\beta^2v^2-4\varepsilon(2-\alpha-\beta u))^2+4\beta^2v^2(\alpha+\beta u-1+\varepsilon)^2}+(\alpha+\beta u-1-\varepsilon)^2-\beta^2v^2-4\varepsilon(2-\alpha-\beta u)}{2}}\\
    \label{eq:implicit_alpha_sqrt}
    \Leftrightarrow&
    \scriptstyle{(\alpha+\beta u-1-\varepsilon)^2+\beta^2v^2+4\varepsilon(2-\alpha-\beta u)=\sqrt{((\alpha+\beta u-1-\varepsilon)^2-\beta^2v^2-4\varepsilon(2-\alpha-\beta u))^2+4\beta^2v^2(\alpha+\beta u-1+\varepsilon)^2}}\\
    \nonumber\Rightarrow&\,\,\scriptstyle{((\alpha+\beta u-1-\varepsilon)^2+\beta^2v^2+4\varepsilon(2-\alpha-\beta u))^2=((\alpha+\beta u-1-\varepsilon)^2-\beta^2v^2-4\varepsilon(2-\alpha-\beta u))^2+4\beta^2v^2(\alpha+\beta u-1+\varepsilon)^2}\\
    \nonumber\Leftrightarrow&\,\,(\alpha+\beta u-1-\varepsilon)^2(\beta^2v^2+4\varepsilon(2-\alpha-\beta u))=v^2(\alpha+\beta u-1+\varepsilon)^2\\
    \nonumber\Leftrightarrow&\,\,\beta^2v^2((\alpha+\beta u-1-\varepsilon)^2-(\alpha+\beta u-1+\varepsilon)^2)+4\varepsilon(\alpha+\beta u-1-\varepsilon)^2(2-\alpha-\beta u)=0\\
    \nonumber\Leftrightarrow&\,\,\varepsilon(\alpha+\beta u-1-\varepsilon)^2(2-\alpha-\beta u)-\beta^2v^2\varepsilon(\alpha+\beta u-1)=0\\ \label{eq:implicit_alpha_cub}
    \Leftrightarrow&\,\,(\alpha+\beta u-1-\varepsilon)^2(2-\alpha-\beta u)-\beta^2v^2(\alpha+\beta u-1)=0
\end{align}
    The resulting polynomial $p(\alpha,\varepsilon)$ in equivalence \eqref{eq:implicit_alpha_cub} is cubic in variable $\alpha$, which guarantees the existence of a real root for every $\varepsilon\in\mathds{R}$. Additionally, observe that
$$\tfrac{\partial p}{\partial \alpha}(\alpha,\varepsilon)=(\alpha+\beta u-1-\varepsilon)(5-3\alpha-2\beta u+\varepsilon)-\beta^2v^2\Rightarrow\tfrac{\partial p}{\partial \alpha}(1-\beta u,0)=-\beta^2v^2\not=0,$$
    so that the Implicit Function Theorem allows us to find a local $\mathscr{C}^1$ solution $(\acrit{\beta,1}(\varepsilon),\varepsilon)$ defined near point $(1-\beta u,0)$. Note that when $\alpha=1-\beta u$, $\varepsilon=0$, the following expression yields
$$(\alpha+\beta u-1-\varepsilon)^2+\beta^2v^2+4\varepsilon(2-\alpha-\beta u)=\beta^2v^2>0;$$
    therefore, equations \eqref{eq:implicit_alpha_og} through \eqref{eq:implicit_alpha_cub} are actual equivalences along solution $(\acrit{\beta,1}(\varepsilon),\varepsilon)$. It remains to see that critical value $\alpha=\acrit{\beta,1}(\varepsilon)$ satisfies $\fmax{1}=0$. To verify this, by once again using the Implicit Function Theorem, derivative $\acrit{\beta,1}'(\varepsilon)$ is found to be
$$\dfrac{\mathrm{d}\acrit{\beta,1}}{\mathrm{d}\varepsilon}(\varepsilon)=-\dfrac{\tfrac{\partial p}{\partial \varepsilon}}{\tfrac{\partial p}{\partial \alpha}}(\acrit{\beta,1}(\varepsilon),\varepsilon)=\dfrac{2(\acrit{\beta,1}(\varepsilon)+\beta u-1-\varepsilon)(2-\acrit{\beta,1}(\varepsilon)-\beta u)}{(\acrit{\beta,1}(\varepsilon)+\beta u-1-\varepsilon)(5-3\acrit{\beta,1}(\varepsilon)-2\beta u+\varepsilon)-\beta^2v^2},$$
    thus we get $\acrit{\beta,1}'(0)=0$ and $\acrit{\beta,1}'(\varepsilon)\not=0$ for small values of $\varepsilon>0$. This in particular implies the quadratic growth formula \eqref{eq:order}. By the definition of derivatives, this implies for small enough values of $\varepsilon$ that
$$\abs{\dfrac{\acrit{\beta,1}(\varepsilon)-1+\beta u}{\varepsilon}}<1,$$
    and thus, for $\varepsilon>0$, we get bounds for the growth of function $\acrit{\beta,1}$, which bounds \eqref{eq:order_bound}. These inequalities imply that term $\acrit{\beta,1}(\varepsilon)+\beta u-1-\varepsilon$ is negative for small enough values of $\varepsilon>0$, so Equation \eqref{eq:fmax_zero} is verified, and thus $\alpha=\acrit{\beta,1}(\varepsilon)$ satisfies $\fmax{1}=0$. Then for $\beta\in(0,\tfrac{1}{u})$, $\varepsilon>0$, $\alpha=\acrit{\beta,1}(\varepsilon)$ we have $\nmax{1}=i\pmax{1}$, which by formulae \eqref{eq:lead2}, equivalence \eqref{eq:implicit_alpha_sqrt},  and inequality \eqref{eq:order_bound} reduces to
$$\pmax{1}=\dfrac{\beta v}{2}+\dfrac{\mathrm{sgn}(\beta v)}{2}\sqrt{\beta^2v^2+4\varepsilon(2-\alpha-\beta u)}.$$
    Recall that we take $v>0$ in the non-real case. Then having $\pmax{1}=0$ for $\beta\in(0,\tfrac{1}{u})$, $\varepsilon>0$, $\alpha=\acrit{\beta,1}(\varepsilon)$ would imply $\alpha=2-\beta u$, which is false for small enough values of $\varepsilon>0$ as $\acrit{\beta,1}(0)=1-\beta u$, therefore making $\nmax{1}$ a pure imaginary, non-real number. Partially differentiating $\fmax{\mu}$ with respect to $\alpha$ yields
$$\dfrac{\partial\fmax{1}}{\partial\alpha}(\alpha,\beta,\varepsilon)=\dfrac{1}{2}+\dfrac{1}{4}\mbox{\tiny{$\tfrac{\tfrac{((\alpha+\beta u-1-\varepsilon)^2-\beta^2v^2-4\varepsilon(2-\alpha-\beta u))(\alpha+\beta u-1+\varepsilon)+2\beta^2v^2(\alpha+\beta u-1+\varepsilon)}{\sqrt{((\alpha+\beta u-1-\varepsilon)^2-\beta^2v^2-4\varepsilon(2-\alpha-\beta u))^2+4\beta^2v^2(\alpha+\beta u-1+\varepsilon)^2}}+\alpha+\beta u-1+\varepsilon}{\sqrt{\tfrac{\sqrt{((\alpha+\beta u-1-\varepsilon)^2-\beta^2v^2-4\varepsilon(2-\alpha-\beta u))^2+4\beta^2v^2(\alpha+\beta u-1+\varepsilon)^2}+(\alpha+\beta u-1-\varepsilon)^2-\beta^2v^2-4\varepsilon(2-\alpha-\beta u)}{2}}}$}}.$$
    In order to evaluate the preceding expression at $\beta\in(0,\tfrac{1}{u})$, $\varepsilon>0$ and $\alpha=\acrit{\beta,1}(\varepsilon)$, from equivalence \eqref{eq:implicit_alpha_og}, the denominator in the second term reduces to
$$\sqrt{\tfrac{\sqrt{((\alpha+\beta u-1-\varepsilon)^2-\beta^2v^2-4\varepsilon(2-\alpha-\beta u))^2+4\beta^2v^2(\alpha+\beta u-1+\varepsilon)^2}+(\alpha+\beta u-1-\varepsilon)^2-\beta^2v^2-4\varepsilon(2-\alpha-\beta u)}{2}}=\abs{\alpha+\beta u-1-\varepsilon}.$$
    This term is non-zero by inequality \eqref{eq:order_bound} as its argument is negative. The other square root term in $\tfrac{\partial\fmax{1}}{\partial\alpha}$ is similarly simplified by equivalence \eqref{eq:implicit_alpha_sqrt}. Thus the partial derivative $\tfrac{\partial\Phi}{\partial\alpha}(\acrit{\beta,1}(\varepsilon),\beta,\varepsilon)$ yields
$$\dfrac{1}{2}+\dfrac{1}{2}\left(\dfrac{\acrit{\beta,1}(\varepsilon)+\beta u-1+\varepsilon}{1+\varepsilon-\acrit{\beta,1}(\varepsilon)-\beta u}\right)\left(\dfrac{(\acrit{\beta,1}(\varepsilon)+\beta u-1-\varepsilon)^2+\beta^2v^2}{(\acrit{\beta,1}(\varepsilon)+\beta u-1-\varepsilon)^2+\beta^2v^2+4\varepsilon(2-\acrit{\beta,1}(\varepsilon)-\beta u)}\right).$$ 
    Inequality \eqref{eq:order_bound} guarantees that the second addendum in the previous expression is positive, which concludes that $\tfrac{\partial\fmax{\mu}}{\partial\alpha}$ is positive as well at 
    $\beta\in(0,\tfrac{1}{u})$, $\varepsilon>0$, $\alpha=\acrit{\beta,1}(\varepsilon)$. Finally, restrict the domain for $\acrit{\beta,1}(\varepsilon)$ so that all of the previous assumptions (its definition through the Implicit Function Theorem, and its quadratic growth) hold. 
\end{proof}

\cexistbetafirst*

\begin{proof}
    Proceed analogously as in Lemma \ref{lem:cexist}. The coefficient for $\beta^3$ in the polynomial term in equivalence \eqref{eq:implicit_alpha_cub} is $-u(u^2+v^2)$, which is non-zero by hypothesis. Then it is possible to find real solutions to \eqref{eq:implicit_alpha_cub} for $\beta$. Partially differentiating the associated polynomial with respect to $\beta$ now yields 
$$\tfrac{\partial p}{\partial \beta}(\beta,\varepsilon)=u(\alpha+\beta u-1-\varepsilon)(5-3\alpha-3\beta u+\varepsilon)-2\beta v^2(\alpha+\beta u-1)-\beta^2uv^2,$$
    which implies $\tfrac{\partial p}{\partial \beta}(\tfrac{1-\alpha}{u},0)=(\alpha-1)v^2\not=0$, so we may use the Implicit Function Theorem to find a $\mathscr{C}^1$ solution to \eqref{eq:implicit_alpha_cub}, denoted $\bcrit{\alpha,1}(\varepsilon)$, and satisfying $\bcrit{\alpha,1}(0)=\tfrac{1-\alpha}{u}$. In can be analogously seen that $\bcrit{\alpha,1}'(0)=0$, therefore guaranteeing inequality
    $$\abs{\dfrac{\bcrit{\alpha,1}(\varepsilon)-\tfrac{1-\alpha}{u}}{\varepsilon}}<\dfrac{1}{u}$$
    for small enough values of $\varepsilon>0$, which in turn implies bounds~\eqref{eq:order_bound_beta}.
    As before, from this last inequality it follows that term $\alpha+\bcrit{\alpha,1}(\varepsilon)-1-\varepsilon$ is negative, therefore satisfying $\fmax{1}=0$. Then $\nmax{1}=i\pmax1$, with $\pmax{1}$ the same expression as in Lemma \ref{lem:cexist}, only now for $\alpha\in(0,1)$, $\varepsilon>0$ and $\beta=\bcrit{\alpha,1}(\varepsilon)$. Partial derivative $\tfrac{\partial\fmax{1}}{\partial\beta}$ is given by
$$\dfrac{u}{2}+\dfrac{1}{4}\mbox{\tiny{$\tfrac{\tfrac{((\alpha+\beta u-1-\varepsilon)^2-\beta^2v^2-4\varepsilon(2-\alpha-\beta u))(u(\alpha+\beta u-1+\varepsilon)-2\beta v^2)+2\beta^2uv^2(\alpha+\beta u-1+\varepsilon)+2\beta v^2(\alpha+\beta u-1+\varepsilon)^2}{\sqrt{((\alpha+\beta u-1-\varepsilon)^2-\beta^2v^2-4\varepsilon(2-\alpha-\beta u))^2+4\beta^2v^2(\alpha+\beta u-1+\varepsilon)^2}}+u(\alpha+\beta u-1+\varepsilon)-2\beta v^2}{\sqrt{\tfrac{\sqrt{((\alpha+\beta u-1-\varepsilon)^2-\beta^2v^2-4\varepsilon(2-\alpha-\beta u))^2+4\beta^2v^2(\alpha+\beta u-1+\varepsilon)^2}+(\alpha+\beta u-1-\varepsilon)^2-\beta^2v^2-4\varepsilon(2-\alpha-\beta u)}{2}}}$}}.$$
    Once again we simplify this expression through equations \eqref{eq:implicit_alpha_og} and \eqref{eq:implicit_alpha_sqrt}, and inequality \eqref{eq:order_bound_beta}. This yields
    $$\dfrac{u}{2}+\dfrac{1}{2}\mbox{\tiny{$\dfrac{((\alpha+\beta u-1-\varepsilon)^2-\beta^2v^2-4\varepsilon(2-\alpha-\beta u))(u(\alpha+\beta u-1+\varepsilon)-\beta v^2)+2\beta^2uv^2(\alpha+\beta u-1+\varepsilon)+\beta v^2(\alpha+\beta u-1+\varepsilon)^2}{(1+\varepsilon-\alpha-\beta u)((\alpha+\beta u-1-\varepsilon)^2+\beta^2v^2+4\varepsilon(2-\alpha-\beta u))}$}}.$$
    Now observe that
    $$\lim_{\varepsilon\to0^*}((\alpha+\bcrit{\alpha,1}(\varepsilon) u-1-\varepsilon)^2-\bcrit{\alpha,1}(\varepsilon)^2v^2-4\varepsilon(2-\alpha-\bcrit{\alpha,1}(\varepsilon) u))(u(\alpha+\bcrit{\alpha,1}(\varepsilon) u-1+\varepsilon)-\bcrit{\alpha,1}(\varepsilon) v^2)=\dfrac{(1-\alpha)^3v^2}{u^3}$$
    which is positive as we imposed $\alpha\in(0,1)$. Thus the second addendum in the reduced expression for $\tfrac{\partial\fmax{1}}{\partial\beta}(\alpha,\bcrit{\alpha,1}(\varepsilon),\varepsilon)$ is positive for small enough values of $\varepsilon>0$, therefore guaranteeing that $\tfrac{\partial\fmax{1}}{\partial\beta}$ is positive as well for $\alpha\in(0,1)$, $\varepsilon>0$ and $\beta=\bcrit{\alpha,1}(\varepsilon)$, which concludes the proof.
\end{proof}

\rdomfirst*

\begin{proof}
    For the sake of simplicity we introduce the following auxiliary definition. Observe that for $\varepsilon=0$ the solutions to condition \eqref{eq:rquad} are directly computed as $\alpha+\beta\mu-1$ and 0. To analytically distinguish them, especially at $\alpha=1-\beta\mathrm{Re}(\mu)$, note that one of them varies linearly on $\alpha$, and the other one is constant. Thus we define the $\lMIN$-$J_0$-\emph{eigenvalue} associated to $\mu_j$, denoted as $\lmin{j}$ as the continuous solution to \eqref{eq:rquad} which satisfies $\tfrac{\partial\mathrm{Re}(\lmin{j})}{\partial\alpha}(\alpha,0)=0$, and the $\lMAX$-$J_0$-\emph{eigenvalue} associated to $\mu_j$, denoted as $\lmax{j}$ as the continuous solution to \eqref{eq:rquad} which satisfies $\tfrac{\partial\mathrm{Re}(\lmax{j})}{\partial\alpha}(\alpha,0)\not=0$. A general correspondence between $\lmin{j}$, $\lmax{j}$ and $\nmin{j}$, $\nmax{j}$ cannot be ascertained for arbitrary values of $\alpha$, $\beta$, and $\varepsilon$.
    
    Denote $\mu_j=u_j+iv_j$ for $j\in\{1,\ldots,k\}$. From Lemma \ref{lem:rexist} we know the explicit definition for the critical value function, $\acrit{\mu,1}(\varepsilon)=1+\varepsilon-\beta\mu_1$. By Lemmata \ref{lem:conj} and \ref{lem:rexist}, eigenvalues $\nmin{1}$ and $\nmax{1}$ are conjugate non-real numbers for small enough values of $\varepsilon>0$, $\alpha=\acrit{\beta,1}(\varepsilon)$. We split the remaining elements in $\sigma(J_0)$ into $\lMIN$-$J_0$- and $\lMAX$-$J_0$-eigenvalues, as previously defined, $\lmin{j}(\alpha,0)=0$, $\lmax{j}(\alpha,0)=\alpha+\beta\mu_j-1$. Then, for $j\in\{2,\ldots,k\}$, it is clear that $\mathrm{Re}(\lmax{j}(\acrit{1}(0),0))=(1-\beta u_1)+\beta u_j-1<0$. By continuity, this inequality is preserved for small enough values of $\varepsilon>0$, which takes care of the result for roots $\lmax{j}$. Now denote $\lmin{j}$ in their Cartesian form, $\lmin{j}=\varphi_j+i\psi_j$. Then the polynomial in \eqref{eq:rquad} is equivalently written
$$(\varphi_j+i\psi_j)^2+(1+\varepsilon_j-\alpha_j-\beta u_j-i\beta v_j)(\varphi_j+i\psi_j)+\varepsilon(2-\beta u_j-i\beta v_j-\alpha),$$
    which after developing yields
$$\varphi_j^2-\psi_j^2+2i\varphi_j\psi_j+(1+\varepsilon-\alpha-\beta u_j)\varphi_j+\beta v_j\psi_j+\varepsilon(2-\beta u_j-\alpha)+i((1+\varepsilon-\alpha-\beta u_j)\psi_j-\beta v_j\varphi_j-\beta\varepsilon v_j). $$
    The problem of finding the roots for the previously defined quadratic polynomial is equivalent to finding points in the zero level set of function $(F_j,G_j)$, where
\begin{equation}\label{eq:levelset}
    \begin{split}
        F_j(\alpha,\varepsilon,\varphi_j,\psi_j)&=\varphi_j^2-\psi_j^2+(1+\varepsilon-\alpha-\beta u_j)\varphi_j+\beta v_j\psi_j+\varepsilon(2-\beta u_j-\alpha),\\
        G_j(\alpha,\varepsilon,\varphi_j,\psi_j)&=2\varphi_j\psi_j+(1+\varepsilon-\alpha-\beta u_j)\psi_j-\beta v_j\varphi_j-\beta \varepsilon v_j.
    \end{split}
\end{equation}
    The Jacobian determinant of function $(F_j,G_j)$ for subsystem $(\varphi_j,\psi_j)$ is readily computed as
$$\abs{\dfrac{\partial(F_j,G_j)}{\partial(\varphi_j,\psi_j)}}=\abs{\begin{array}{cc}
    2\varphi_j+1+\varepsilon_j-\alpha-\beta u_j & -2\psi_j+\beta v_j  \\
    2\psi_j-\beta v_j & 2\varphi_j+1+\varepsilon-\alpha-\beta u_j 
\end{array}}=(2\varphi_j+1+\varepsilon-\alpha-\beta u_j)^2+(2\psi_j-\beta v_j)^2.$$
    Recall that $u_1>u_j$ for every $j\geqslant2$. Then, at solution $\varepsilon=\varphi=\psi=0$, $\alpha=1-\beta u_1$ and for every $j\geqslant2$, this determinant reduces to
$$\dfrac{\partial(F_j,G_j)}{\partial(\varphi_j,\psi_j)}=\beta^2(u_1-u_j)^2+\beta^2v_j^2\not=0,$$
    and therefore, by virtue of the Implicit Function Theorem, variables $\varphi_j$ and $\psi_j$ can be expressed as $\mathscr{C}^1$ functions $\Phi_j(\alpha,\varepsilon)$ and $\Psi_j(\alpha,\varepsilon)$ inside the zero level set for values $(\alpha,\varepsilon)$ close enough to $(1-\beta\mu_1,0)$. Moreover, it is possible to compute their derivatives with respect to $\varepsilon$ by
$$\dfrac{\partial(\Phi_j,\Psi_j)}{\partial\varepsilon}=-\dfrac{\partial(F_j,G_j)}{\partial(\varphi_j,\psi_j)}^{-1}\left(\dfrac{\partial(F_j,G_j)}{\partial\varepsilon}+\dfrac{\partial(F_j,G_j)}{\partial\alpha}\dfrac{\mathrm{d}\acrit{\beta,1}}{\mathrm{d}\varepsilon}\right),$$
    where the second term inside the parentheses is given by the Chain Rule as $\alpha=\acrit{\beta,1}(\varepsilon)$ varies as a function of $\varepsilon>0$. Evaluating at $(1-\beta\mu_1,0,0,0)$ yields
$$\dfrac{\partial(\Phi_j,\Psi_j)}{\partial\varepsilon}(1-\beta u_1,0)=\dfrac{-1}{\beta^2(u_1-u_j)^2+\beta^2v_j^2}\left(\begin{array}{c}
    \beta u_1-\beta u_j+\beta^2(u_1-u_j)^2+\beta^2v_j^2-0 \\
    \beta v_j-0  
\end{array}\right),$$
    from whence it follows that $\tfrac{\partial\Phi_j}{\partial\varepsilon}<0$ in a vicinity of $(1-\beta\mu_1,0)$. As $\Phi_j(\alpha,0)=0$, this implies that $\mathrm{Re}(\lmin{j})<0$ for small enough values of $\varepsilon>0$ and $\alpha=\acrit{\beta,1}(\varepsilon)$, for every $j\geqslant2$, thus proving the desired result.
\end{proof}

\cdomfirst*

\begin{proof}
    Recall that  Lemma \ref{lem:conj} guarantees that $\cj{\nu}_1^-=\nmin{2}$ and $\cj{\nu}_1^+=\nmax{2}$. We proceed analogously as in Lemma \ref{lem:rdom}, which proves the result for associated eigenvalues $\nmin{j}$ and $\nmax{j}$, that is, for $\lmin{j}$ and $\lmax{j}$, for $j\in\{3,\ldots,k\}$. Now we have to prove that $\fmin{1}=\fmin{2}$ is negative as well. First
    observe that the expression for $\pmax{1}$ found in Lemma \ref{lem:cexist} guarantees that $\nmax{1}$ is a $\lMAX$-eigenvalue, therefore making $\nmin{1}$ a $\lMIN$-eigenvalue, $\nmin{1}(\alpha,0)=0+0i$. When applying the Implicit Function Theorem to zero level set-condition \eqref{eq:levelset} for $j\in\{1,2\}$, at solution $(\alpha,\varepsilon,\phi_1,\psi_1)=(1-\beta u_1,0,0,0)$, the Jacobian determinant is seen to be 
    $$\dfrac{\partial(F_1,G_1)}{\partial(\varphi_1,\psi_1)}(1-\beta u_1,0)=\beta^2(u_1-u_1)^2+\beta^2v_1^2=\beta^2v_1^2\not=0$$
    since $\mu_1=u_1+iv_1\not\in\mathds{R}$, so that $v_1\not=0$, and $\beta>0$. Recall from \eqref{eq:order} that $\acrit{\beta,1}'(0)=0$ in the non-real case. Now, differentiating $\Phi_1$ with respect to $\varepsilon$ yields
    $$\dfrac{\partial(\Phi_1,\Psi_1)}{\partial\varepsilon}(1-\beta u_1,0)=\dfrac{-1}{\beta^2v_1^2}\left(\begin{array}{c}
    \beta^2v_1^2-\acrit{1}'(0)(0+\varphi_1) \\
    \beta v_1-\acrit{1}'(0)\psi_1  
\end{array}\right)=\left(\begin{array}{c}
     -1  \\
     -\tfrac{1}{\beta v_1} 
\end{array}\right),$$
    and thus $\tfrac{\partial\Phi_1}{\partial\varepsilon}(1-u_1,0)=-1<0$, from whence the result follows.
\end{proof}

\coeffgenfirst*

\begin{proof}
    By Theorem \ref{theo:datta}, $\vo{v}$ can be rescaled to satisfy $\vo{v}^t\vo{w}=0$, $\cj{\vo{v}}^t\vo{w}=2$ which, by virtue of Equations (\ref{eq:reigvec}) and (\ref{eq:leigvec}), translate to
\begin{equation}\label{eq:eigvecs_rescal}
\left(1-\tfrac{\varepsilon}{\varepsilon^2+\abs{\nmax{1}}^2}\right)\vo{v}_x^t\vo{w}_x=0,\ \ 
    \left(1-\tfrac{\varepsilon}{(\varepsilon-\nmax{1})^2}\right)\cj{\vo{v}}_x^t\vo{w}_x=2,
\end{equation}

Determining $b$ requires knowing the first three directional derivatives of vector field $\vo{f}$. Given any direction $\vo{r}=(\vo{r}^t_x\vert\vo{r}_y^t)^t$, the first one is given by
\begin{align*}
    &\sum_{k=1}^N
    \left(\begin{array}{c}                          
    -\delta_{jk}+B^{\alpha}_{jk}S'\left(\sum_{l=1}^NB_{jl}^{\alpha}x_l\right)\\
         \hline
    \varepsilon\delta_{jk}
    \end{array}\right)^{l\in\{1,\ldots,N\}}_{(\vo{0}_N,\vo{0}_N),0}(\vo{r}_x)_k 
    + \sum_{k=1}^N
    \left(\begin{array}{c}
         -\delta_{jk}\\
         \hline
        -\varepsilon\delta_{jk}
    \end{array}\right)^{j\in\{1,\ldots,N\}}_{(\vo{0}_N,\vo{0}_N),0}(\vo{r}_y)_j 
    \\
    &=\sum_{k=1}^N
    \left(\begin{array}{c}
         -\delta_{jk}+B_{jk}^{0}S'(0)\\
         \hline
        \varepsilon\vo{e}_j
    \end{array}\right)^{j\in\{1,\ldots,N\}}(\vo{r}_x)_j + \sum_{k=1}^N
    \left(\begin{array}{c}
        -\vo{e}_j\\
        \hline
        -\varepsilon\vo{e}_j
    \end{array}\right)(\vo{r}_y)_j = 
    \left(\begin{array}{c}
         \vo{r}_x+B\vo{r}_x\\
         \hline
        \varepsilon\vo{r}_x
    \end{array}\right) - \left(\begin{array}{c}
        \vo{r}_y\\
        \hline
        \varepsilon\vo{r}_y\\
    \end{array}\right) \\
    & = \left(\begin{array}{c}
        -\vo{r}_x-\vo{r}_y+B\vo{r}_x\\
        \hline
        \varepsilon(\vo{r}_x-\vo{r}_y)\\
    \end{array}\right);
\end{align*}
now we compute the second derivative at directions $\vo{r}=(\vo{r}^t_x\vert\vo{r}_y^t)^t$, $\vo{s}=(\vo{s}^t_x\vert\vo{s}_y^t)^t$, dividing it into four terms according to the mixed partial differentiation required:
\begin{align*}
    &\scriptstyle{\sum_{k,m}
    \tfrac{\partial\ \ \ \ }{\partial x_m}\left(\begin{array}{c}                          
    -\delta_{jk}+B^{\alpha}_{jk}S'\left(\sum_{l=1}^NB_{jl}^{\alpha}x_l\right)\\
         \hline
    \varepsilon\delta_{jk}
    \end{array}\right)^{j\in\{1,\ldots,N\}}_{(\vo{0}_N,\vo{0}_N),0}(\vo{r}_x)_k(\vo{s}_x)_m
    + \sum_{k,m}
    \tfrac{\partial\ \ }{\partial x_m}
    \left(\begin{array}{c}
         -\delta_{jk}\\
         \hline
        -\varepsilon\delta_{jk}
    \end{array}\right)^{j\in\{1,\ldots,N\}}_{(\vo{0}_N,\vo{0}_N),0}(\vo{r}_y)_k(\vo{s}_x)_m}\\
    +&\scriptstyle{\sum_{k,m}
    \tfrac{\partial\ \ }{\partial y_m}\left(\begin{array}{c}            -\delta_{jk}+B^{\alpha}_{jk}S'\left(\sum_{l=1}^NB_{jl}^{\alpha}x_l\right)\\
         \hline
    \varepsilon\delta_{jk}
    \end{array}\right)^{j\in\{1,\ldots,N\}}_{(\vo{0}_N,\vo{0}_N),0}(\vo{r}_x)_k(\vo{s}_y)_m 
    + \sum_{k,m}
    \tfrac{\partial\ \ }{\partial y_m}
    \left(\begin{array}{c}
         -\delta_{jk}\\
         \hline
        -\varepsilon\delta_{jk}
    \end{array}\right)^{j\in\{1,\ldots,N\}}_{(\vo{0}_N,\vo{0}_N),0}(\vo{w}_y)_k(\vo{s}_y)_m}\\
    &=\sum_{k,m}
    \left(\begin{array}{c}    B^{\alpha}_{jk}B^{0}_{jm}S''\left(0\right)\\
         \hline
    \vo{0}_N
    \end{array}\right)^{j\in\{1,\ldots,N\}}(\vo{r}_x)_k(\vo{s}_x)_m+\vo{0}_{2N}=\sum_{k,m}\left(\begin{array}{c}    \vo{0}_N\\
         \hline
    \vo{0}_N
    \end{array}\right)=\vo{0}_{2N}.
\end{align*}
Finally, we compute the third derivative at directions $\vo{r}=(\vo{r}^t_x\vert\vo{r}_y^t)^t$, $\vo{s}=(\vo{s}^t_x\vert\vo{s}_y^t)^t$, $\vo{t}=(\vo{t}^t_x\vert\vo{t}_t^t)^t$, which consists of eight terms in a similar manner of how we have previously proceeded. However, as seen is the second derivative calculations, any mixed differentiation of the form
$$\dfrac{\partial^2\vo{f}}{\partial x_k\partial y_m},$$
as well as high order $y_j$-derivatives are all equal to zero. By means of Clairaut's theorem, we are therefore only required to compute the mixed $x_j$ term, which reduces to
\begin{align*}
    &\sum_{k,m,n}
    \tfrac{\partial\ \ \, }{\partial x_n}\left(\begin{array}{c}         B^{\alpha}_{jk}B^{\alpha}_{jm}S''\left(\sum_{l=1}^NB_{jl}^{\alpha}x_l\right)\\
         \hline
    0
    \end{array}\right)^{j\in\{1,\ldots,N\}}_{(\vo{0}_N,\vo{0}_N),0}(\vo{r}_x)_k(\vo{s}_x)_m(\vo{t}_x)_n\\
    =&\sum_{k,m,n}
    \left(\begin{array}{c}         B_{jk}B_{jm}B_{jn}S'''\left(0\right)\\
         \hline
    0
    \end{array}\right)^{j\in\{1,\ldots,N\}}(\vo{r}_x)_k(\vo{s}_x)_m(\vo{t}_x)_n=S'''(0)\left(\begin{array}{c}
         B\vo{r}_x\odot B\vo{s}_x\odot B\vo{t}_x  \\
         \hline
         \vo{0}_N 
    \end{array}\right).
\end{align*}
 Recall that the definition for coefficient $b$ requires us to evaluate at eigenvalues $\vo{r}=\vo{s}=\vo{w}$ and $\vo{t}=\cj{\vo{w}}$. By remembering equivalence \eqref{eq:rquad_eq} and the fact that $B=\acrit{\beta,1}(\varepsilon)I_N+\beta A$, we observe that
\begin{align*}
    B\vo{w}_x&=\acrit{\beta,1}(\varepsilon)\vo{w}_x+\beta A\vo{w}_x=\acrit{\beta,1}(\varepsilon)\vo{w}_x+(1-\acrit{\beta,1}(\varepsilon)-\nmax{1}+\tfrac{\varepsilon}{\varepsilon-\nmax{1}})\vo{w}_x=(1-\nmax{1}+\tfrac{\varepsilon}{\varepsilon-\nmax{1}})\vo{w}_x,\\
    B\cj{\vo{w}}_x&=\acrit{\beta,1}(\varepsilon)\cj{\vo{w}}_x+\beta A\cj{\vo{w}}_x=\acrit{\beta,1}(\varepsilon)\cj{\vo{w}}_x+\cj{(1-\acrit{\beta,1}(\varepsilon)-\nmax{1}+\tfrac{\varepsilon}{\varepsilon-\nmax{1}})}\cj{\vo{w}}_x=(\cj{1-\nmax{1}+\tfrac{\varepsilon}{\varepsilon-\nmax{1}}})\cj{\vo{w}}_x,
\end{align*}
and therefore,
$$\left(\begin{array}{c} B\vo{w}_x\odot B\vo{w}_x\odot B\cj{\vo{w}}_x \\
         \hline
         \vo{0}_N 
    \end{array}\right)=(1-i\abs{\lambda}+\tfrac{\varepsilon}{\varepsilon-i\abs{\lambda}})\abs{1-i\abs{\lambda}+\tfrac{\varepsilon}{\varepsilon-i\abs{\lambda}}}^2\left(\begin{array}{c}
         \vo{w}_x\odot\vo{w}_x\odot\cj{\vo{w}}_x \\
         \hline
         \vo{0}_N 
    \end{array}\right).$$
We then substitute the expression for the third directional derivative in the definition for coefficient $b$. We obtain
\begin{align*}
    b &= \tfrac{1}{16}\mathrm{Re}\left(\ip{(\vo{v}_x^t\vert\vo{v}_y^t)^t}{(1-\nmax{1}+\tfrac{\varepsilon}{\varepsilon-\nmax{1}})\abs{1-\nmax{1}+\tfrac{\varepsilon}{\varepsilon-\nmax{1}}}^2S'''(0)((\vo{w}_x\odot\vo{w}_x\odot\cj{\vo{w}}_x)^t\vert\vo{0}_N^t)^t}\right) \\
    &= \tfrac{1}{16}\abs{1-\nmax{1}+\tfrac{\varepsilon}{\varepsilon-\nmax{1}}}^2S'''(0)\mathrm{Re}\left(\left(1-\nmax{1}+\tfrac{\varepsilon}{\varepsilon-\nmax{1}}\right)\ip{\vo{v}_x}{\vo{w}_x\odot\vo{w}_x\odot\cj{\vo{w}}_x}+0\right),
\end{align*}
    which concludes the proof.
\end{proof}

\critfirst*

\begin{proof}
     We can rescale vector $\vo{w}_x$ so that $\abs{(\vo{w}_x)_j}=1$ for every $j\in\{1,\ldots,N\}$. This reduces formula~\eqref{eq:hopf_coeff_b} to
     \begin{align*}
     b&=\tfrac{1}{16}\abs{1-\nmax{1}+\tfrac{\varepsilon}{\varepsilon-\nmax{1}}}^2S'''(0)\sum_{j=1}^N\mathrm{Re}\left((1-\nmax{1}+\tfrac{\varepsilon}{\varepsilon-\nmax{1}})(\cj{\vo{v}}_x)_j(\vo{w}_x)_j\abs{(\vo{w}_x)_j}^2\right)\\
     &=\dfrac{1}{16}\abs{1-\nmax{1}+\dfrac{\varepsilon}{\varepsilon-\nmax{1}}}^2S'''(0)\sum_{j=1}^N\mathrm{Re}\left(\left(1-\nmax{1}+\dfrac{\varepsilon}{\varepsilon-\nmax{1}}\right)(\cj{\vo{v}}_x)_j(\vo{w}_x)_j\right)\\
     &=\dfrac{1}{16}\abs{1-\nmax{1}+\dfrac{\varepsilon}{\varepsilon-\nmax{1}}}^2S'''(0)\mathrm{Re}\left(\left(1-\nmax{1}+\dfrac{\varepsilon}{\varepsilon-\nmax{1}}\right)\cj{\vo{v}}_x^t\vo{w}_x\right)\\
     &=\dfrac{1}{8}\abs{1-i\abs{\lambda}+\dfrac{\varepsilon}{\varepsilon-\nmax{1}}}^2S'''(0)\mathrm{Re}\left(\left(1-\nmax{1}+\dfrac{\varepsilon}{\varepsilon-\nmax{1}}\right)\left(\dfrac{(\varepsilon-\nmax{1})^2}{(\varepsilon-\nmax{1})^2-\varepsilon}\right)\right).
     \end{align*}
     As locally odd sigmoid function $S$ is assumed to satisfy $S'''(0)<0$ (see Section \ref{sec:math}), it is clear that the sign of coefficient $b$ is opposite to that of the real-part term in the previous expression, whose argument $z$ is given by
     \begin{equation}\label{eq:cbsign}
    z=\left(1-\nmax{1}+\dfrac{\varepsilon}{\varepsilon-\nmax{1}}\right)\left(\dfrac{(\varepsilon-\nmax{1})^2}{(\varepsilon-\nmax{1})^2-\varepsilon}\right).  
    \end{equation}
     To determine the sign of $\mathrm{Re}(z)$ we need to substitute the values for $\abs{\nmax{1}}$, as before, by splitting the proof into a real and non-real case. Recall that in the real case we get the expression $\abs{\nmax{1}}=\sqrt{\varepsilon(1-\varepsilon)}$ from Equation \eqref{eq:rmodul}. Then
    $$z=(1+\varepsilon)\dfrac{\varepsilon-\tfrac{1}{2}-i\sqrt{\varepsilon(1-\varepsilon)}}{\varepsilon-1-i\sqrt{\varepsilon(1-\varepsilon)}}=(1+\varepsilon)\dfrac{\tfrac{1}{2}(1-\varepsilon)-\tfrac{3}{2}i\sqrt{\varepsilon(1-\varepsilon)}}{(1-\varepsilon)^2+\varepsilon(1-\varepsilon)},$$
    and thus $\mathrm{Re}(z)>0$ for $\varepsilon\in(0,1)$, therefore $b<0$. For the non-real case we directly evaluate at the singular limit,
\begin{align*}
    \lim_{\varepsilon\to0^+}&\mathrm{Re}\left(\left(1-i\abs{\lambda}+\dfrac{\varepsilon}{\varepsilon-i\abs{\lambda}}\right)\left(\dfrac{(\varepsilon-i\abs{\lambda})^2}{(\varepsilon-i\abs{\lambda})^2-\varepsilon}\right)\right) 
    =\mathrm{Re}\left(\left(1-iv+\dfrac{0}{0-iv}\right)\left(\dfrac{(0-iv)^2}{(0-iv)^2-0}\right)\right)\\
    =&\mathrm{Re}\left(\left(1-iv\right)\left(\dfrac{-v^2}{-v^2}\right)\right)=\mathrm{Re}(1-iv)=1>0.     
\end{align*}
This concludes that term $z$ from \eqref{eq:cbsign} has positive real part, thus proving that $b$ is negative for small enough values of $\varepsilon>0$, therefore making the bifurcation supercritical, which is what we wanted to prove.
\end{proof}